\documentclass[11pt]{article}
\usepackage{amsmath, amssymb, amsfonts, latexsym, xcolor, mathrsfs, longtable, enumitem,amsthm}

\theoremstyle{plain}
\newtheorem{theorem}{Theorem}[section]
\newtheorem{lemma}{Lemma}[section]
\newtheorem{corollary}{Corollary}[section]
\newtheorem{proposition}{Proposition}[section]

\theoremstyle{definition}
\newtheorem{definition}{Definition}[section]
\newtheorem{example}{Example}[section]

\theoremstyle{remark}
\newtheorem{remark}{Remark}[section]

\makeatletter
\renewcommand\section{\@startsection{section}{1}{\z@}%
                                   {-5.5ex \@plus -1ex \@minus -.2ex}%
                                   {1.3ex \@plus.2ex}%
                                   {\normalfont\normalsize\bfseries}}
\renewcommand{\@seccntformat}[1]{\csname the#1\endcsname\ }

\renewcommand\subsection{\@startsection{subsection}{2}{\z@}%
                                   {-2.5ex \@plus -1ex \@minus -.2ex}%
                                   {1.3ex \@plus.2ex}%
                                   {\normalfont\normalsize}}
\makeatother

\def\RR{\mathbb{R}}
\def\P{\mathscr{P}}
\def\L{\mathcal{L}}
\def\M{\mathcal{M}}
\def\A{\mathcal{A}}
\def\B{\mathcal{B}}

\def\P{\mathcal{P}}
\def\V{\mathcal{V}}

\def\R{\mathcal{G}}

\def\la{\langle}
\def\ra{\rangle}
\def\raa{\rightarrow}

\DeclareMathOperator{\inte}{int}
\DeclareMathOperator{\ri}{ri}
\DeclareMathOperator{\core}{core}
\DeclareMathOperator{\icr}{icr}
\DeclareMathOperator{\sqri}{sqri}
\DeclareMathOperator{\lin}{lin}
\DeclareMathOperator{\wsup}{WSup}
\DeclareMathOperator{\winf}{WInf}
\DeclareMathOperator{\wmax}{WMax}
\DeclareMathOperator{\wmin}{WMin}
\DeclareMathOperator{\cl}{cl}
\DeclareMathOperator{\bd}{bd}
\DeclareMathOperator{\co}{co}
\DeclareMathOperator{\cone}{cone}
\DeclareMathOperator{\aff}{aff}
\DeclareMathOperator{\dom}{dom}

\DeclareMathOperator{\epi}{epi}

\DeclareMathOperator{\gr}{gr}
\DeclareMathOperator{\exepi}{\mathfrak{E} {\rm{pi}}}

\def\nplus{\boxplus}
\allowdisplaybreaks

\begin{document}

\vskip-1cm
\title{A Perturbation  Approach to Vector Optimization Problems: Lagrange and Fenchel-Lagrange Duality}


\author{
		N. Dinh\thanks{CONTACT N. Dinh. Email:  ndinh@hcmiu.edu.vn.} \thanks{Department of Mathematics, International University Vietnam National University-Ho Chi Minh City, Ho Chi Minh City, Vietnam,  e-mail: ndinh@hcmiu.edu.vn; Vietnam National University, Ho Chi Minh City,  Vietnam.} 
		\and D.H. Long\thanks{VNUHCM - University of Science, District 5, Ho Chi Minh city, Vietnam; Tien Giang University, Tien Giang town, Vietnam; e-mail: danghailong@tgu.edu.vn.}}

\date{}

\maketitle


\begin{abstract}  In this paper we study the  general minimization vector  problem (P), 
concerning a   perturbation  mapping, defined in   locally  convex  Hausdorff  topological vector spaces where the  ``WInf" stands for the weak infimum with respect to   
an ordering  generated by a convex cone  $K$.  Several representations of the epigraph of the conjugate mapping of the perturbation mapping   are established. From these,  
variants vector Farkas lemmas are then proved. 
Armed with these basic tools,  the {\it dual} and the so-called {\it loose   dual problem}  of  (P) are defined, and then 
stable strong duality results between these pairs of primal-dual problems are established.  The results just obtained are then applied to a general class (CCCV) of composed vector optimization problems with 
cone-constrained.  For this classes of problems, four perturbation mappings   are suggested. Each of these mappings   yields  several forms of vector Farkas lemmas and two forms of dual problems for (CCVP). Concretely, one of the suggested perturbation mapping give rises to well-known  {\it Lagrange} and {\it loose Lagrange dual problems} for (CCVP) while each of the three others,  yields two kinds of Fenchel-Lagrange dual problems for (CCVP). Stable strong duality for these pairs of primal-dual problems are proved. Several special cases of (CCVP) are also considered at the end of the paper, including:  vector composite problems (without constraints), cone-constrained vector problems, and scalar composed problems. The results obtained in this papers when specified to the two concrete mentioned vector  problems  go some  Lagrange duality results appeared recently,  and also lead to new results on stable strong Fenchel-Lagrange duality results, which, to the best knowledge of the authors, appear for the first time in the literature. For the scalar composite problems, our  results go back or extend  many Lagrange and Fenchel-Lagrange duality results in the literature.  
\medskip

\noindent
{\bf Keywords}\quad
{Vector optimization problems $\cdot$ perturbation mappings $\cdot$ perturbation approach $\cdot$ vector Farkas lemmas  $\cdot$ stable strong duality  for vector problems}

\noindent
{\bf Mathematics Subject Classification (2010)}\quad
49N15 $\cdot$  90C25 $\cdot$ 90C29 $\cdot$ 90C31 $\cdot$ 90C46 $\cdot$ 90C48

\end{abstract}




\section{Introduction}
\label{sect1} 
 We are  concerned with the general vector optimization problem 
\begin{align*}
({\rm P}^L)&\qquad \mathop{\winf}\limits_{x\in X}[\Phi(x,0_Z)-L(x)],
\end{align*}
associated with  a  perturbation  mapping $\Phi\colon X\times Z\to Y\cup\{+\infty_Y\}$, and a linear operator  $ L \in \V\subset  \L(X, Y)$,  where $X,Y,Z$ are  locally  convex  Hausdorff  topological vector spaces (in brief, lcHtvs) and ``WInf" stands for the weak infimum w.r.t.    an ordering  generated by a convex cone  $K\subset Y$.  {The problem  $({\rm P}^L) $ or $(P)$ (when $L =0_\L\in \L(X, Y)$) is very general vector optimization problems which  includes models considered in e.g., the book \cite{BGW09}, and also classes of practical problems (see, e.g.,  \cite{CDLP20,Khan-Tammer-Zali}, and references therein).}
  
Conjugate duality has {been}  used to study duality for scalar, vector problems, and also for set-valued optimization problems by many authors, see, for instance, \cite{Bot2010,DNV-08,DVV-14,DMVV-Siopt,ET76,Z02,BGWMIA09,JSDL05,Burachik0Jeya-Wu}  for scalar and robust (scalar)  problems, and \cite{BGW09,BGW097nw,BGOPT11,DGLL17,DGLMJOTA16,DL2017,Tanino92,Tammer07,Tanino79}  for vector and set-valued problems. 
 The perturbation approach  was also  recently developed  for  {vector-valued}  optimization problems  in \cite{BGW09} {to study the duality for these  problems. }
   
In this  paper  we use  the perturbation approach to  vector optimization problems to establish {stable strong  duality results} for  $({\rm P}^L)$.   We first  prove some  representations of the epigraph of   conjugate mapping of $\Phi (., 0_Z)$,  $\epi \Phi(.,0_Z)^\ast$.  and then prove transcriptions (necessary and sufficient conditions) of  general ``vector inequality" of the form: 
  \begin{equation}\label{VI}
\Phi(x,0_Z)    \notin -\inte K,\; \forall x\in X . 
\end{equation}
These transcriptions can be considered as  {generalizations of Farkas lemmas for the vector system   \eqref{VI}.}  Lastly,  for $\emptyset \ne \V \subset  \L(X, Y)$, and each $L\in \V$, we introduce 
 the \emph{ dual problem} and the \emph{loose  dual problem} of $({\rm VP}^L)$, respectively, 
 \begin{align*}
 ({\rm D}^L)& \qquad \mathop{\wsup}\limits_{{T\in \L_\Phi}} [-\Phi^* (L,T)], \\
({\rm D}_\ell^L)&\qquad  \mathop{\wsup}\limits_{{T\in \L_\Phi^+}}[-\Phi^* (L,T)], 
\end{align*}
where $\L_{\Phi}:=    \{T\in \L(Z,Y): \exists   L\in \L(X,Y)\   \textrm{s.t.} \         (L,T)\in \dom \Phi^\ast \}$  and  $\L_\Phi^+:= \L_\Phi \cap  \L_+(S, K)$,  where    $\L_+(S, K) = \{ T \in \L(X, Z) :  T(S) \subset K \}$, and $S$ is a convex cone in $Z$.  Several results on the $\V$-stability of  strong duality for the pairs of primal-dual problems, i.e.,  ${\rm (P}^L) - {\rm (D}^L)$ and   ${\rm (P)} - {\rm (D}_{\ell}^L)$ for  any $L \in \V$ are established.

 In order to overcome difficulties arising from   a general approach,  we  use some new notions  such as   "partition-style subsets of $Y$",  {the ordering     "$\preccurlyeq_K $" between such two  sets,}     "extended epigraphs" of conjugate mappings,  and  new operations,  $\uplus$-sum,   between two subsets of $Y$,  and the operation $\nplus$-sum on the collection of  {all} extended epigraphs that just introduced recently by the same authors in    \cite{DL20}.    
 We  {prove}  also a new version of open mapping theorem concerning a map defined  on a product space (all  these   are presented in Section 2).   
 
  The novelty of our work  is three-fold: Firstly, these new tools  permit us to establish variants of representations of the  epigraph of the conjugate mapping, $\epi \Phi^\ast$,  of the perturbation mapping of $\Phi $, which pay the way for proving versions of general vector Farkas lemmas and  also stable strong duality for the pairs (P) - $ ({\rm D}^L) $ and (P) - $ ({\rm D}^L_\ell) $, for which, in concrete  problems give rise to stable strong duality concerning  Fenchel-Lagrange dual problems that probably appeared for the first time for vector optimization problems (see Sections 5 and 6).  In the case where $Y = \RR$ these dual problems go back to  the traditional Fenchel-Lagrange dual problems as in \cite{Bot2010,DNV-08,BGWMIA09,DVN-08} (see Section 6) and this justifies the name  ``Lagrange and Fenchel-Lagrange dual problems" of  (D) and $({\rm D}_\ell)$.   
 Secondly, {our  approach is} based on a  very general perturbational $K$-convex mapping $\Phi$ (not necessarily  continuous) {which} enables a wide-range of applications (i.e., one gets different results with different choices of $\Phi$) not only on vector optimization problems but also on alternative theorems for vector inequalities.
Beside the new results that come from different mappings $\Phi$, the approach   generalizes and unifies the antecedent  works on duality vector optimization problems and vector Farkas lemmas   such as   \cite{CDLP20,DGLL17,DGLMJOTA16,DGLM-Optim17}.

To illustrate the meaning of the results, we consider at the end of the paper   a class of  
the composed vector   problem  with a cone constraint:
   \begin{align*}
({\rm CCVP})&\qquad \winf \{F(x)+(\kappa \circ H)(x): \quad x\in C,\; G(x)\in-S\}.  
\end{align*}   
a special case  of this problem is the cone-constrained vector problem 
\begin{align*}
({\rm VP})&\qquad \winf \{F(x):x\in C,\; G(x)\in -S\}
\end{align*}
which is the common model many practical problems  in science and engineering, for instance,  the problem of designing of fiber distributed data interface computer networks, the problem of designing of a cross-current multistage extraction process    (see e.g., \cite{CDLP20}, \cite[Chapter 13]{Jahn-vector}).  
 Four different   perturbation mappings  for (CCVP): $\Phi_1, \Phi_2, \Phi_3$, and $\Phi_4$  are proposed, and  generalized vector Farkas lemmas for the systems associated to (CCVP),  {\it Lagrange and Fenchel-Lagrange  dual problems} (eight in all, two problems for each $\Phi_i$) then follow together with 
 stable strong   duality results corresponding to these dual problems. 
 In Section 6, some more specific problems in the class $({\rm CCVP}) $ are considered, including the (VP) and the scalar composite problems. Here,  with the specifications of  $\Phi_1, \Phi_2, \Phi_3$, and $\Phi_4$ to each of concrete problems, results in Sections 5 give rise to the ones which cover/extend the corresponding known ones in the literature, and also propose several new results.  An illustrative  example showing one of the  new and specific feature of the our results, which never appeared before,   is as follows:  with (VP), $\Phi_3$ (the same  observation for $\Phi_2, \Phi_4$)  
 leads to the representation: 
\begin{align}\label{A1}
 \epi (F+I_A)^\ast
&=    \Psi\Big(\exepi F^*\boxplus\exepi I_C^*\boxplus\!\!\!\!\! \bigcup_{T\in\L_+(S,K)}\!\!\!\!\!\exepi  ( T\circ G)^*\Big), 
\end{align}
where $A := C \cap G^{-1}(-S)$,  $I_C: X \rightarrow Y^\bullet$  is the indicator mapping of $C$, and $\exepi F^*$ is the extended epigraph of {$F^\ast$.}  When $Y = \RR$, the representation collapses to 
\begin{equation} \label{A2}
\epi (f+i_A)^*=\epi f^*+\epi i_C^*+\!\!\bigcup\limits_{z^*\in S^+}\!\!\epi(z^*\circ G)^*,
\end{equation} 
which  appeared  in many works (see,  e.g.,   \cite{Bot2010,DNV-08,DVV-14,BGWMIA09,JSDL05,DVN-08} and references therein). This shows that the representation of of epigraph of conjugate functions  as in \eqref{A2} is extend, for the first time,  to the one for epigraph of conjugate mappings. 
   
The paper is organized as follows: In Section 2 we firstly  present some basic notations and known results which are necessary in the sequent.  Secondly, we 
recall some  notions to be used along this paper, which  include  notions of $(Y. K)$-partition style sets, extended epigraphs of conjugate mappings,  operations   (as, $\uplus$-sum, $\nplus$-sum),  and  orderings together  with their basic properties which were recently introduced by the same authors in \cite{DL20}. Besides,  a new version of an  extended open mapping theorem is also proved. .
  In Section 3 several representations of the epigraph of the conjugate mapping  $ \Phi(., 0_Z)^\ast $   are established, and based on these, several necessary and sufficient conditions of the ``vector inequality" of the form \eqref{VI} are proved. 
  These  are actually general vector Farkas lemmas. In fact, it is shown  that in many concrete situations (see Sections 5, 6)  they generalize,   cover    the  known ones   in the literature,  or give rise to new ones versions of (vector/scalar)  Farkas lemmas, even when  $Y = \RR$.     Section 4 presents  the main results of the paper:  
   necessary and sufficient conditions for $\V$-stable strong duality of for (VP).  More precisely,  we show that the  strong duality holds for the pairs $({\rm P}^L) - {\rm (D}^L)$ and   $({\rm P}^L) - {\rm (D}_{\ell}^L)$ for each $L \in \V$. Sections 5 and 6   are left as   illustration examples for  the meaning of   the results obtained in Sections 3, 4. Concretely, in Section 5, we consider the composed vector   problem (CCVP) with a cone constraint   and some special cases of this are considered in Section 6. We show in these two last sections that the results obtained in this paper when specified to these concrete problems give rise to several consequences due to different choices of $\Phi$. Some among them  cover known ones in the literature and  some are  new even when coming back to the case where $Y = \RR$. Some similar or  rather technical proofs of some auxiliary results are left to the  Appendix at the end of the paper. 
  
 \section{Preliminaries, First Results,  and  Generalized Open Mapping Theorem}
\label{section2}
\subsection{Prelimainaries}    

{Let} $X,Y,Z$ be lcHtvs with their topological dual spaces denoted by $X^{\ast },Y^{\ast }$ and $Z^{\ast }$, respectively, and all are equipped with  the weak*-topology.
For a set $U\subset X$, we denote by $\inte U,\ \cl U,\ \bd U$, $\ \co U$, $\ \lin U$, $\ \aff U,\ \cone U$ the \emph{interior},  the \emph{closure}, the \emph{boundary}, the \emph{convex hull}, the \emph{linear hull}, the \emph{affine hull}, and the \emph{cone hull} of $U$, respectively. The  \emph{intrinsic core} of $U$,  
 \emph{relative interior} of $U$ are defined respectively as \cite{Bot2010,Z02} 
\begin{eqnarray*} 
\icr U&:=&\{x\in X:\forall x'\in \aff U,\; \exists \delta>0  \textrm{ such \ that \ }  x+\lambda x'\in U,\; \forall \lambda \in [0,\delta]\},\\
\ri U&:=& \{x\in \aff U: \exists V,  \textrm{ the neighborhood of } x  \textrm{ such \ that \ }  (V\cap \aff U)\subset U\}. 
\end{eqnarray*} 
 It is well-known that $\core U \subset \sqri U \subset \icr U \subset U$ and when $X$ is a finite dimension space, one has $\icr U =\sqri U = \ri U$.

Assume that $X_0$ is a topological subspace of $X$. For $A\subset X_0$, denote by $\inte_{X_0} A$ the interior of $A$ w.r.t. the topology induced in  $X_0$.  
Given two subsets $A$ and $B$ of a topological space,
one says that $A$ is \emph{closed regarding} $B$ if $B\cap \cl A=B\cap A$ (see, e.g., \cite[p. 56]{Bot2010}).


$\bullet$ {\it Weak Ordering  Generated by  a Convex Cone. }   
Let $K$ be a proper,  closed, and convex cone in $Y$ with nonempty interior, i.e.,  $\inte K\neq \emptyset $.  
It is worth observing that 
$K+\inte K=\inte K,  
$
or equivalently, 
\begin{equation}
\left. 
\begin{tabular}{c}
$y\in K$ \\ 
$y+y^{\prime}\notin\inte K$%
\end{tabular}
\right\} \ \Longrightarrow\ y^{\prime}\notin\inte K.  \label{int-plus}
\end{equation}

We define  a \emph{weak ordering} in $Y$ generated by $K$ as follows: for all $y_1, y_2 \in Y$, 
\begin{equation}
y_{1}<_{K}y_{2}\; \Longleftrightarrow \; y_{1}-y_{2}\in -\inte K. 
\label{eq_weak_ordering_strict}
\end{equation}
In $Y$ we sometimes   also consider an usual   \emph{ ordering} generated by the cone $K$, $\leqq_K$, which is defined by $y_1\leqq_Ky_2$ if and only if $y_1-y_2\in -K, $ for $y_1, y_2 \in Y$.

\begin{lemma} {\rm \cite[Lemma 2.1]{CDLP20} }   \label{pro_1a} The following assertions hold: 
\begin{itemize}
\item[$\rm(i)$] For all $y,y'\in Y$ and $k_0\in\inte K$, there is $\mu>0$ 
such that  $y -\mu k_0<_K  y'$,  
\item[$\rm(ii)$] For all finite collection $\{y_i\}_{i=1}^n\subset Y$, there exists $y_0, y'_0\in Y$ such that 
$$y'_0<_K y_i<_K y_0,\quad \forall i\in\{1,\ldots, n\}.$$ 
\end{itemize}
\end{lemma}

We enlarge $Y$ by attaching a \emph{greatest element} $+\infty _{Y}$ and a 
\emph{smallest element} $-\infty _{Y}$ with respect to $<_{K}$, which do not
belong to $Y$, and we denote $Y^{\bullet }:=Y\cup \{-\infty _{Y},+\infty
_{Y}\}$.We also denote that $Y^\infty:=Y\cup \{+\infty
_{Y}\}$. By convention,  $-\infty _{Y}<_{K}y<_{K}+\infty _{Y}$ for any $y\in Y
$ and, for  $M\subset Y$,    
\begin{gather}
-\!(\!+\!\infty _{Y}\!)=-\!\infty _{Y}, \quad-\!(\!-\!\infty _{Y})=+\!\infty _{Y}, 
\quad
(\!+\!\infty _{Y}\!)\!+\!y=y\!+\!(\!+\!\infty _{Y}\!)=+\!\infty _{Y},\; \forall y\!\in\! Y^\infty
 \notag \\
(\!-\!\infty _{Y}\!)\!+\!y=y\!+\!(\!-\!\infty _{Y}\!)=-\!\infty _{Y},\; \forall y\!\in\! -Y^\infty ,    \label{2.4}\\ 
M\!+\!\{\!-\!\infty _{Y}\!\}=\{\!-\!\infty _{Y}\!\}\!+\!M=\{\!-\!\infty _{Y}\!\},\ \ 
 M\!+\!\{\!+\!\infty_{Y}\!\}=\{\!+\!\infty _{Y}\!\}\!+\!M=\{\!+\!\infty _{Y}\!\}. \notag
 \end{gather}%
The sums $(-\infty _{Y})+(+\infty _{Y})$ and $(+\infty _{Y})+(-\infty _{Y})$
are not considered in this paper.


The following notions are the key ones and  will be used throughout  the paper.

\begin{definition}{\rm (\cite[Definition 7.4.1]{BGW09}, \cite{Tanino92})} \label{def1} Let $M\subset Y^{\bullet }$. 

$\left(\rm a\right) $ An element $\bar{v}\in Y^{\bullet }$
is said to be a \emph{weakly infimal element} of $M$ if for all $v\in M$ we
have $v\not<_{K}\bar{v}$ and if for any $\tilde{v}\in Y^{\bullet }$ such
that $\bar{v}<_{K}\tilde{v}$, then there is some $v\in M$ satisfying $%
v<_{K}\tilde{v}$. The set of all weakly infimal elements of $M$ is denoted
by $\winf M$ and is called the \emph{weak infimum} of $M$.\medskip

$\left(\rm b\right) $ An element $\bar{v}\in Y^{\bullet }$ is said
to be a \emph{weakly supremal element} of $M$ if for all $v\in M$ we have $%
\bar{v}\not<_{K}v$ and if for any $\tilde{v}\in Y^{\bullet }$ with  $%
\tilde{v}<_{K}\bar{v}$, then there exists some $v\in M$ satisfying $\tilde{v}%
<_{K}v$. The set of all supremal elements of $M$ is denoted by $\wsup M$ and
is called the \emph{weak supremum} of $M$.\medskip

$\left(\rm c\right) $ The \emph{weak minimum} of $M$ is the set $%
\wmin M=M\cap \winf M$ and its elements are the \emph{weakly minimal elements%
} of $M$.\medskip

$\left(\rm d\right) $   The \emph{weak maximum} of $M$ is the set $\wmax M=M\cap \wsup M$ and its elements are the \emph{weakly maximal elements} of $M$.
\end{definition}

The  next properties  related to the sets weak infimum, weak minimum, weak supremum, weak maximum of a subset $M$ of $Y^\bullet$ are  traced out from  \cite{BGW09,DGLL17,DL2017}.   


\begin{proposition}\cite[Proposition 2.1]{DL20}    
\label{pro_decomp}
Let $\emptyset \ne M,N\subset Y^\bullet$. One has

$\mathrm{(i)}$\  $\wsup M\ne \{+\infty_Y\}$ if and only if there exists $ Y\setminus  (M-\inte K)  \ne \emptyset$, 

$\mathrm{(ii)}$\  For all $ y\in Y$, $\wsup (y+M)=y+\wsup M$. 

$\mathrm{(iii)}$\   $\wsup (\wsup M+\wsup N)= \wsup  (M+\wsup N)=\wsup (M+N).$ 

\noindent
Assume further that $M\subset Y$ and  $\wsup M\ne \{+\infty_Y\}$ then it holds:

$\mathrm{(iv)}$\  $\wsup M-\inte K=M-\inte K$,

$\mathrm{(v)}$\  The following decomposition\footnote{Here, by the term ``decomposition" we mean the sets  in the right-hand side of the equality are disjoint.}  of $Y$ holds 
  $$Y=(M-\inte K)\cup \wsup M\cup (\wsup M + \inte K), $$ 

$\mathrm{(vi)}$\  $\wsup M=\cl(M-\inte K)\setminus (M-\inte K)$.


$\mathrm{(vii)}$\  If $0_Y\in N\subset -K$ then $\wsup (M+N)=\wsup M$. 

In particular, one has $\wsup (M-K)=\wsup (M-\bd K)=\wsup M$. 
\end{proposition}

\begin{remark}
\label{rem_1hh}
It is clear  that $\winf M=-\wsup(-M)$ for all $M\subset Y^\bullet$ and  so, Proposition \ref{pro_decomp}  holds true also when  $\wsup$, $+\infty_Y$, $K$, and $\inte K$ are replaced by $\winf$, $-\infty_Y$, $-K$, and $-\inte K$, respectively. For instance, corresponding to $\mathrm{(v)}$, one has:

$\mathrm{(v')}$\  $\winf M = \cl(M    + \inte K)\setminus (M   + \inte K)$.
\end{remark}


$\bullet$ {\it Conjugate Mappings  of Vector-Valued Functions. } 
{Let $F \colon X\rightarrow Y^{\bullet }$.  The \emph{domain},    and the \emph{$K$-epigraph} of $F $  are defined  respectively by 
\begin{align*}
\dom F &:=\{x\in X:F (x)\neq +\infty _{Y}\}, \ \ 
\epi_{K}F: =\{(x,y)\in X\times Y:y\in F (x)+K\}.
\end{align*}
The mapping  $F $ is   \emph{proper} if $\dom F \neq \emptyset $ and $-\infty
_{Y}\notin F (X)$. It is \emph{$K$-convex} if $\epi_K F$ is convex in  $X \times Y$, 
 and it is  \emph{$K$-epi closed} if $\epi_K F$ is closed in the product space $X\times Y$   \cite[Definition 5.1]{Luc}, \cite{Bot2010}.    
When $Y =   \mathbb{R}$ and $K = \mathbb{R}_+$,   $\mathbb{R}_+$-epi closed property collapses to the lower semicontinuity   of the  real-valued function $F$.  The mapping  $F$ is called   \emph{positively $K$-lower semicontinuous} (lsc, for brief)  if $y^*\circ F$ is lsc for all $y^*\in K^+$ (see \cite{Bot2010,JSDL05})\footnote{This notion was used   in    \cite{Bot2010} and in  \cite{JSDL05}  as  ``star $K$-lower semicontinuous''.}. According  to  \cite[Theorem 5.9]{Luc}, every positively $K$-lsc mapping is $K$-epi closed but the converse is not true. Moreover,   when $Y=\RR$, the three notions lsc, positively $\RR_+$-lsc, and $\RR_+$-epi closed coincide with each other.}

For the space  $\mathcal{L}(X,Y)$  of all
 continuous linear mappings  from $X$ to $Y$, one  equippes with   the   \emph{topology of point-wise convergence}, i.e.,   if  $(L_{i})_{i\in I}\subset \mathcal{L}(X,Y)$ and $L\in \mathcal{L}(X,Y)$, $L_{i}\rightarrow L$ in $\mathcal{L}(X,Y) $
means  $L_{i}(x)\rightarrow L(x)$ in $Y$ for all $x\in X$. The zero of $\mathcal{L}(X,Y)$ is  $0_{\mathcal{L}}$.

Now let $S$ be a non-empty convex cone in $Z$. Recall that the {\it cone of positive operators} (see \cite{AB85,KLS89}) and {\it the cone of weakly positive operators} from $Z$ to $Y$ \cite{DGLL17} are defined respectively
\begin{align*}
\L_+(S,K)&:=\{T\in\L(Z,Y): T(S)\subset K\}, \\
\L_+^w(S,K)&:=\{T\in \L(Z,Y) : T(S)\cap (-\inte K)=\emptyset\}.
\end{align*}
When  $Y = \mathbb{R}$, both these  cones collapse to  $S^+\!\!:=\!\{ z^* \in Z^\ast \!: \!\!\la z^*, s \ra \geq 0, \forall s \in S \}$.

For $T\in \mathcal{L}(Z,Y)$,  $G\colon X\rightarrow Z^\infty$,  we define the composite function $T\circ G\colon X\rightarrow {Y}^\infty$ as:  
\begin{equation*}
(T\circ G)(x):=\left\{ 
\begin{array}{ll}
T(G(x)), & \text{if }G(x)\in Z, \\ 
+\infty _{Y}, & \text{if $G(x)=+\infty _{Z}.$}%
\end{array}%
\right.
\end{equation*}

The following notion of conjugate mapping  is specified from the corresponding  one for set-valued mappings in   \cite[Definition 7.4.2]{BGW09}, \cite[Definition 3.1]{Tanino92}.

\begin{definition} {\rm \cite{BGW09,Tanino92}} \label{def2.3} For  $F\colon X\rightarrow Y^\infty$,  the set-valued mapping $F
^{\ast }\colon \mathcal{L}(X,Y)\rightrightarrows Y^{\bullet }$ defined by 
$F ^{\ast }(L):=\wsup\{L(x)-F (x):x\in X\}$
is called the \emph{conjugate mapping} of $F $. The \emph{$K$-epigraph}  and the \emph{domain} of $F^*$ and  are, respectively,
\begin{align*}
\epi_{K}\!\! F ^{\ast }\!\!:=\!\!\big\{(L,y)\in \mathcal{L}(X,Y)\times Y\! \!\!:\! y\in F
^{\ast }(L)+K\!\big\}, \  \dom\! F ^{\ast }\!\!:=\!\!\big\{L\in \mathcal{L}(X,Y) \!:\!F ^{\ast }(L)\!\neq\!
\! \{+\infty
_{Y}\}\!\big\}. 
\end{align*}
\end{definition}
Throughout this paper, we use the only  cone $K$ in $Y$, and so, for the sake of simplicity,  from now on we will write $\epi F$ and $\epi F^*$
instead of $\epi_KF$ and $\epi_K F^*$, respectively. 
Further more, by   \cite[Lemma {3.5}]{DGLL17}, if  $F\colon X\rightarrow Y^\infty$ is a proper
mapping then $\epi F^{\ast}$ is a closed  (but not necessarily convex) subset of $
\mathcal{L}(X,Y)\times Y.$ Moreover,  for  $(L,y)\in \mathcal{L}(X,Y)\times Y$, one has (see \cite[Proposition 2.10]{DGLMJOTA16})
\begin{equation}\label{epiF*}
(L,y)\in \epi F ^{\ast }\quad\Longleftrightarrow\quad \Big(F (x)-L(x)+y\notin -\inte %
K,\;\forall x\in X\Big). 
\end{equation}
Lastly, for a subset $D\subset X$, the \emph{indicator map} $
I_{D}\colon X\rightarrow {Y}^{\bullet }$ is defined by $I_D (x) = 0_Y $ if $x \in D$ and $I_D(x) = +\infty_Y $, otherwise.   It is worth noting that 
 \cite[Proposition {3.1}]{DGLL17},   
\begin{equation} \label{domI*}
\dom I^*_{-S}=\L_+^w(S,K) .
\end{equation}


\subsection{Structure $(\mathcal{P}_p(Y)^\infty, \preccurlyeq_K, {\uplus})$, extended epigraphs, and  $\nplus$-summation}

Let $\mathcal{P}_0(Y^\bullet)$ be the collection  of all non-empty subsets of $Y^\bullet$.  
The  ordering ``$\preccurlyeq_K$''  on   $\mathcal{P}_0(Y^\bullet)$  is defined   \cite{DL2017}   as, 
 for $M,N \in \mathcal{P}_0(Y^\bullet)$, 
\begin{equation}
\label{eq_6.33a}
M \ \preccurlyeq_K N \ \     \Longleftrightarrow \ \   \left( v\not<_Ku, \; \forall u\in M,\; \forall v\in N\right).  
\end{equation}
Other orderings on $\mathcal{P}_0(Y^\bullet)$ are also proposed in the literature (see, e.g.,  \cite{Khan-Tammer-Zali}).

\begin{proposition}\label{prop_4gg} The following assertions hold 

$\mathrm{(i)}$\ For all  $  M, N \in \mathcal{P}_0(Y^\bullet)\setminus\{\{+\infty_Y\},\{-\infty_Y\}\}$, one has  
\begin{eqnarray*}  M  \preccurlyeq_K N     \ \     &\Longleftrightarrow& \ \      N\cap (M-\inte K)= \emptyset   \ \ 
    \Longleftrightarrow \ \  M\cap (N+\inte K)= \emptyset, 
    \end{eqnarray*}
    \indent $\mathrm{(ii)}$ For all $M\in \mathcal{P}_0(Y^\bullet)$, one has  $\winf M \preccurlyeq_K M \textrm{ and } M  \preccurlyeq_K \wsup M, $

$\mathrm{(iii)}$  $\textrm{If } M\subset N \subset Y^\bullet$ then $\wsup M \preccurlyeq_K \wsup N$, 

$\mathrm{(iv)}$ For all  $M,N\in \mathcal{P}_0(Y^\bullet)$, if $M\preccurlyeq_K N$ then $\wsup M \preccurlyeq_K \winf N.$
\end{proposition}

\begin{proof}
 $\rm(i)-(iii)$  are in \cite[Proposition 3.1]{DL20} while (iv)    follows from  Proposition 3.2  in \cite{DL2017}.      \end{proof}

We say that $U    \subset   Y  $  has a {\it $(Y,K)$-partition style}  if  the following decomposition of $Y$  holds 
\begin{equation} \label{defPp}
Y=(U-\inte K)\cup U\cup (U+\inte K).
\end{equation}  
The collection of all  $(Y,K)$-partition style subsets of $Y$ is     denoted   by $\mathcal{P}_p(Y)$. 
Denote also, 
$\mathcal{P}_p(Y)^\bullet := \mathcal{P}_p(Y) \cup \{ \{ + \infty_Y\}, \{ - \infty_Y\}\}$. 
Clearly that  if  $M\subset Y^\bullet$ then  $\pm\wsup M \in \mathcal{P}_p( Y)^\bullet$,  $\pm\winf M \in \mathcal{P}_p( Y)^\bullet$ and (by \eqref{eq_6.33a}),   for any  $U \in \mathcal{P}_p(Y)$,  one has $U   \preccurlyeq_K \{ + \infty_Y\}$ and $\{ - \infty_Y\}  \preccurlyeq_K U$.    
Moreover,   $\left(\mathcal{P}_p(Y)^\bullet, \preccurlyeq_K\right) $  is an ordered space \cite{DL2017}, i.e.,  $\preccurlyeq_K$ is a   reflexive, anti-symestric, and transitive partial order.

\begin{proposition}\label{pro_5hh} 
{\emph{\cite[Proposition 3.2, Lemma 3.1]{DL20}}}
 Let  $U,V\in \mathcal{P}_p(Y)^\bullet$.     Then 

$\mathrm{(i)}$    If\  $U\subset V$ then $U=V$. 

\noindent Moreover, if   $U,V\subset Y$ then

 $\mathrm{(ii)}$   $U+K=U\cup (U+\inte K)$ and $U-K=U\cup (U-\inte K),$

$\mathrm{(iii)}$ $Y=(U-\inte K)\cup (U+K)=(U-K)\cup (U+\inte K), $

 $\mathrm{(iv)}$    $ U\preccurlyeq_K V\;  \Longleftrightarrow\; V\subset U+K \; 
\Longleftrightarrow\;  U\subset V-K,  $

  $\mathrm{(v)}$   $\wsup U=\winf U=U$.   
\end{proposition}

The   new kind of  ``sum" of two sets   on the collection $\mathcal{P}_p(Y)^\infty :=    \mathcal{P}_p(Y)\cup \{\{+ \infty_Y\}\} $ (called ``WS-sum"), has just introduced in \cite{DL20}, will be one of the main tools for our further  study.

\begin{definition}  \cite{DL20} \label{def_3.1zz}
 For  $U,V\in  \mathcal{P}_p (Y)^\infty  $, the  {\it WS-sum} of $U$ and $V$, denoted by  $U\uplus V$,  is  a set from  $\mathcal{P}_p (Y)^\infty $  and is  defined by: 
\begin{equation}\label{newsum}
 U\uplus V\   :=   \  \wsup (U+V).
\end{equation}
\end{definition} 
Some  properties of the structure $(\mathcal{P}_p(Y)^\infty, \preccurlyeq_K, {\uplus})$  are given in the next proposition.
\begin{proposition} \label{prop_1ab}  
{\emph{\cite[Proposition 3.3]{DL20}}}
Let  $U,V, W\in \mathcal{P}_p(Y)^\infty$, $y\in Y$.  One has

$\mathrm{(i)}$   $U\uplus (-\bd K)=U$,

$\mathrm{(ii)}$  $U\uplus V=V\uplus U$   (commutative), 

$\mathrm{(iii)}$ \  $(U\uplus V)\uplus W= U\uplus (V\uplus W)$ (associative),

$\mathrm{(iv)}$  If $U\preccurlyeq_K V$ then $U\uplus W\preccurlyeq_K V\uplus W$  (compatible of  the sum $\uplus$ with $\preccurlyeq_K$),

 $\mathrm{(v)}$  $y\in (U\uplus V)+K$ if and only if there exists $W\in \mathcal{P}_p(Y)$ such that  $U \preccurlyeq_K W$ and $y\in W\uplus V$.
\end{proposition}



\begin{definition} \cite{DL20}  \label{def_exepi}     $\phantom{x}$
$\textrm{(i)} $	The  {\it $K$-extended epigraph} of the conjugate mapping $F^*$ (of the mapping $F: X \to Y^{\bullet}$)  is defined as  
\begin{equation} \label{extepi}
\mathfrak{E}{\rm pi}  F^*:=\{(L,U)\in \L(X,Y)\times \mathcal{P}_p(Y): L\in \dom F^*,\;  F^*(L)\preccurlyeq_K U\}. 
\end{equation} 
	
$\textrm{(ii)}$ For ${\mathscr{M},\mathscr{N}}\subset  \L(X,Y)\times \mathcal{P}_p(Y)$,   the {\it $\nplus$-sum}  of these two sets is defined as follows: 
\begin{equation}\label{eq_12abc}
{\mathscr{M}\nplus\mathscr{N}:=\{(L_1+L_2, U_1\uplus U_2):(L_1,U_1)\in \mathscr{M},\; (L_2,U_2)\in \mathscr{N}\}.}
\end{equation}
\end{definition}
If  $F, G : X \raa Y^\bullet$ then the $\nplus$-sum of  $\mathfrak{E}{\rm pi }  F^* $ and  $   \mathfrak{E}{\rm pi }  G^*$,  is
$$
\mathfrak{E}{\rm pi }  F^*  \nplus   \mathfrak{E}{\rm pi }  G^*= \{(L_1+L_2, U_1\uplus U_2)\, : \, (L_1,U_1)\in \mathfrak{E}{\rm pi}  F^\ast,\; (L_2,U_2)\in \mathfrak{E}{\rm pi}  G^*   \}. 
$$
We can understood simply  the extended  epigraph $\mathfrak{E}{\rm pi}  F^*$  is   the ``epigraph" of   $F^\ast$ which is  considered as a single valued-mapping  $F^\ast : \L(X, Y) \raa  ( \mathcal{P}_p(Y),\preccurlyeq_K)$ and the ``epigraph" defined in the same way as the one of a real-valued function.  

It is also worth observing that from  the definition of $\nplus $-sum  and  Proposition \ref{prop_1ab},   the     $\nplus$-sum  is  commutative and associative on     $\L(X,Y)\times \mathcal{P}_p(Y)$.  

We      introduce  the set-valued mapping 
\begin{align} \label{Psi}
\Psi \colon \L(X,Y)\times \mathcal{P}_p(Y)&\rightrightarrows \L(X,Y)\times Y     \nonumber   \\
(L,U)\ \ \ \ \ \ &\mapsto \Psi (L,U):=\{L\}\times U.  
\end{align}
The relation between   $ \exepi F^\ast  \subset   \L(X,Y)\times \mathcal{P}_p(Y) $ and   $\epi F^\ast   \subset   \L(X,Y)\times Y $ and some simple properties of  the $\nplus$-sum and the mapping $\Psi$ now   are given  in the next proposition.

\begin{proposition}  \label{rel_epi} 
Let $F\colon X\to Y^\bullet$ be  a proper mapping, and 
  $\mathscr{M}, \mathscr{N}, \mathscr{Q},  \mathscr{M}_i  \subset   \L(X,Y)\times \mathcal{P}_p(Y) $,  for all $i \in I$  ($I$ is an arbitrary index set).  
  Then

  $\mathrm{(i)}$\    $\epi F^*=\Psi (\exepi F^*)$, 
   
 $\mathrm{(ii)}$\  $\mathscr{M} \subset \mathscr{N}   \ \ \ \Longrightarrow \ \ \    \Psi(\mathscr{M}) \subset \Psi (\mathscr{N})$, 
 
   $\mathrm{(iii)}$\  $  \mathscr{M} \subset \mathscr{N}   \ \ \ \Longrightarrow \ \ \      \mathscr{M}  \nplus \mathscr{Q}      \subset \mathscr{N}   \nplus \mathscr{Q}$,  \    
   
  $\mathrm{(iv)}$\    $\bigcup\limits_{i \in I}  \left(\mathscr{M}_i \boxplus\mathcal{N}\right)=\left(\bigcup\limits_{i \in I}  \mathscr{M}_i \right) \boxplus\mathcal{N}$,  
  
  $\mathrm{(v)}$\       $\Psi \big(\bigcup\limits_{i \in I}  \mathscr{M}_i \big)\ \  = \ \  \bigcup\limits_{i \in I} \Psi\big(\mathscr{M}_i\big)$.

\end{proposition}

\begin{proof} 
 $\mathrm{(i)}$ is  \cite[Proposition 3.1]{DL20} while the others are easy from the definitions of $\Psi$ and  $ \nplus$-sum. 
\end{proof} 


\begin{remark}
\label{rem_2dd}
Noting that when $Y=\mathbb{R}$ and $K=\mathbb{R}_+$,  then $\mathcal{P}_p(Y)={\mathbb{R}}$ while the pre-order ``$\preccurlyeq_K$'' (see  \eqref{eq_6.33a})  ``$\uplus$-sum''    (see \eqref{newsum}) become the normal order ``$\le$'' and the normal sum ``$+$'' on the set of extended real numbers, respectively.
Hence, $\L(X,Y)\times \mathcal{P}_p(Y)=X^*\times {\mathbb{R}}$ and the $\boxplus$-sum collapses to  the usual Minkowski sum of two subsets in $X^*\times {\mathbb{R}}$.
Both $K$-epigraph and $K$-extended epigraph of  conjugate mappings  collapse to their  usual  epigraphs in the sense of  convex analysis. The mapping $\Psi$ (defined by \eqref{Psi})    then reduces to  the identical mapping of $X^*\times {\mathbb{R}}$. In  other words, if ${\mathscr{M}, \mathscr{N}}
\subset  X^*\times {\RR}$, one has
{$\Psi(\mathscr{M} \nplus \mathscr{N})=  \Psi (\mathscr{M} + \mathscr{N}) = \mathscr{M} + \mathscr{N}.$}
\end{remark}




\subsection{An Extended Open Mapping Theorem}

Let $\tilde X,\tilde Y$ be lcHtvs
and  let  $\mathcal{G}\colon \tilde X\rightrightarrows \tilde Y$ be a multifunction. The {\it domain } and the {\it graph} of $\mathcal{G}$ are, respectively, 
\begin{gather*}
\dom \mathcal{G}:=\{x\in \tilde X: \mathcal{G}(x)\ne \emptyset\},\ \ 
\gr \mathcal{G}:=\{(x,y)\in \tilde X\times \tilde Y: x\in\dom \mathcal{G} \textrm{ and }y\in \mathcal{G}(x)\}.
\end{gather*}
The multifunction  $\mathcal{G}$ is {\it convex} ({\it closed}, respectively) if $\gr \mathcal{G}$ is a convex (closed, respectively) subset of $\tilde X\times \tilde Y$ \cite[p.13]{Z02}.

\begin{lemma}
{\rm\cite[Lemma 1.3.1]{Z02}}
\label{lem_Zalinescu}
Let $\mathcal{G}\colon \tilde X\rightrightarrows \tilde Y$ be a closed and convex multifunction and $\tilde x_0\in \tilde X$. Assume that $\tilde X$ is  a first countable and complete space\footnote{Recall that a lcHtvs $\tilde X$ is {\it complete} if every  Cauchy net in $\tilde X$ is convergent.}. Then
$$\bigcap_{U\in\mathcal{N}(\tilde x_0)}\inte (\cl \mathcal{G}(U))\subset\bigcap_{U\in \mathcal{N}(\tilde x_0)} \inte \mathcal{G}(U),$$
where $\mathcal{N}(\tilde x_0)$ is the collection  of all neighborhoods of $\tilde x_0$ in $\tilde X$.
\end{lemma}

We are  now turning   back to the lcHtvs $X, Y, Z$ and the non-empty convex cones $K, S$  as in the beginning of this section. Denote 
 by $\pi$ the canonical projection from $X\times Z$ to $Z$, i.e., $\pi (x,z) := z$ for all $(x,z) \in X \times Z$. We need  an extended  open mapping theorem which will be used in the proof of Theorem   \ref{cor_nonasymptotic_representing_epi_3} below, whose proof is rather technical (and based on Lemma \ref{lem_Zalinescu}) and will left to the Appendix.

\begin{theorem}[Extended open mapping theorem]
\label{lem_epiclosed}
Let  $\Phi \colon X\times Z\to Y^\infty$ be a $K$-convex and $K$-epi closed mapping,  $x_0\in X$ with $(x_0,0_Z)\in \dom \Phi$, and 
let   $Z_0=\lin (\pi(\dom\Phi))$. Let further,  $U_0$    and $V_0$  be  neighborhoods  of     $x_0$ and $\Phi (x_0,0_Z)$ in $X$ and \ $Y$, respectively.
 Assume that $X, Y$  are the   first countable  and complete spaces,   $Z_0$ is a barreled space\footnote{A lcHtvs $\tilde X$ is {\it barreled} if every absorbing, convex and closed subset of $\tilde X$ is a neighborhood of the origin of $\tilde X$ \cite[p.9]{Z02}}, and that $0_Z\in \icr (\pi (\dom \Phi))$. Then 
\begin{equation}
\label{eq_2.8d}
0_Z\in \inte_{Z_0} \{z\in Z: \Phi (x,z)\leqq_K y \textrm{ for some } x\in U_0 \textrm{ and } y\in V_0\}.
\end{equation} 
\end{theorem}

\begin{proof} (see Appendix A)  \end{proof}



\section{Conjugate  of     Perturbation Mapping  and   Generalized Vector Inequalities  }
\label{section4}


Let $X,Y$, $Z$  be  lcHtvs and  $K, S$ be  non-empty convex cones in $Y$ and $Z$, respectively, with  $\inte K\ne \emptyset$,   
as in Section \ref{section2}, and let    $\pi :X\times Z \to Z$ be the   canonical   projection from $X\times Z$ to $Z$, i.e., $\pi (x, z) = z$ for all $(x,z) \in X \times Z$.  
  
   
 \subsection{Epigraphs of Conjugate of Perturbation Mappings}   
 Consider a {\it perturbation  mapping  $\Phi\colon X\times Z\to Y^\infty$ }
with  its    conjugate mapping $\Phi^\ast : \L(X,Y)\times \L(Z,Y)  \rightrightarrows      Y^\bullet$.  
The conjugate mapping of $\Phi(.,0_Z)\colon X\to Y^\infty$ will be denoted by   $\Phi (.,0_Z)^\ast$  (instead of $(\Phi (.,0_Z))^\ast$). 
 
  In this section, we will  establish   variant    representations of   $\epi \Phi (.,0_Z)^\ast$ in terms of  $\epi \Phi^\ast (.,T)$ with   $T \in  {\rm proj}_{\L(Z,Y)}\dom \Phi^\ast$, which pay the way for the characterizations of general vector inequalities, and also for the construction of Lagrange and Fenchel-Lagrange dual problems.  
  
     For the sake of  convenience,   the following assumptions/conventions which will be used in the rest of the paper.

$\bullet$ Assume that {\it   $\Phi$ is a proper $K$-convex mapping}  (except for the Proposition \ref{pro_incluepi}), 

$\bullet$ Set  
 \begin{equation} \label{Z0}
 Z_0:=\lin (\pi(\dom\Phi)),   
 \end{equation} 
 and assume   that  $0_Z\in \pi( \dom \Phi)$, which yields    $Z_0=\aff (\pi (\dom \Phi))$.

\indent $\bullet$ Recall  that  
for $z_1, z_2 \in Z$,  $z_{1}\leqq _{S}z_{2}$  means that  $ z_{1}-z_{2}\in -S$, 

\indent $\bullet$ We also enlarge $Z$ by attaching a greatest element $+\infty _{Z}$ and a
smallest element $-\infty _{Z}$, which do not belong to $Z$, and define $%
Z^{\bullet }:=Z\cup \{-\infty _{Z},+\infty _{Z}\}$ and $Z^\infty:=Z\cup \{+\infty _{Z}\}$. In $Z^{\bullet }$ we
adopt the same  conventions as in \eqref{2.4}.

\indent $\bullet$ Let us set 
\begin{eqnarray}
 \L_{\Phi} \!\!\!   &:=& \!\!\!   \{T\in \L(Z,Y): \exists   L\in \L(X,Y)\   \textrm{s.t.} \  (L,T)\in \dom \Phi^\ast \}= 
{\rm proj}_{\L(Z,Y)}\dom \Phi^\ast,    \label{Lphi} \\
 \L_{\Phi}^+\!\!\!     &:=&\!\!\!     \L_\Phi\cap\L_+(S,K), \label{Lphi+}\\
\mathcal{M} \!\!\!    &:=&\!\!\!     \bigcup_{T\in \L_\Phi}\epi \Phi^*(.,T),\qquad\mathcal{M}_+:=\bigcup_{T\in \L_\Phi^+}\epi \Phi^*(.,T).
\label{eq:14_nwewewE}
\end{eqnarray}



\begin{proposition}
\label{pro_incluepi}
It holds\ \ 
{$\epi \Phi(.,0_Z)^\ast \supset \mathcal{M}\supset\mathcal{M}_+.$}
\end{proposition}

\begin{proof}  The second  inclusion   $\mathcal{M}\supset\mathcal{M}_+$   
is trivial. 
So, only the first one  needs to prove. 
Take  $(L,y)\in \bigcup_{T\in {\L_\Phi}}\epi \Phi^*(.,T)$.
 Then, there is  $T\in \L_\Phi$ such that $(L,y)\in \epi \Phi^*(.,T)$, or equivalently, 
$(L,T,y)\in  \epi \Phi^\ast$, and then, by   \eqref{epiF*},    one gets 
$\Phi(x,z)-L(x)-T(z)+y\notin -\inte K$ for all $(x,z)\in X\times Z.$ 
In particular, taking  $z = 0_Z$, one has 
$\Phi(x,0_Z)-L(x) +y\notin -\inte K$ for all $x\in X, $
which, again by  \eqref{epiF*},    $(L,y)\in {\epi\, } \Phi (., 0_Z)^\ast$ and the proof is complete. \end{proof}


Consider the following conditions: 

\begin{tabular}{c | c}
$(C_{0})$ &  
\begin{minipage}{0.8\textwidth}
$\exists \tilde x \in X :  (\tilde x,0_Z)\in \dom \Phi  \textrm{ and }  \Phi(\tilde x,z)\leqq_K\Phi (\tilde x,0_Z), \;\forall z\in-S $, 
\end{minipage}
\end{tabular}

\begin{tabular}{c | c}
$(C_1)$ &  
\begin{minipage}{0.8\textwidth}
$\forall L\in \L(X,Y)$, $\exists y_L\in Y$,  $\exists V_L \in \mathcal{N}(0_Z)$   such that \\
$\forall z\in  V_L\cap Z_0,\; \exists x\in X: \Phi(x,z)-L(x)\leqq_K y_L, $
\end{minipage}
\end{tabular}

\begin{tabular}{c | c}
$(C_2)$ &  
\begin{minipage}{0.8\textwidth}
$\exists (\widehat x,\widehat y)\in X\times Y$, $\exists \widehat V \in \mathcal{N}(0_Z) $ such that $ \Phi(\widehat x,z)\leqq_K \widehat y,\; \forall z \in \widehat V\cap Z_0,$
\end{minipage}
\end{tabular}

\begin{tabular}{c | c}
$(C_3)$ &  
\begin{minipage}{0.8\textwidth}
$\exists \widehat x\in X$   such that  $\Phi_{| Z_0}(\widehat x,.)$ is continuous  at $0_Z$,
\end{minipage}
\end{tabular}

\begin{tabular}{c | c}
$(C_4)$ &  
\begin{minipage}{0.7\textwidth}
 $\dim Z_0<\infty$ and $0_Z\in \ri (\pi(\dom \Phi))$,
\end{minipage}
\end{tabular}

\begin{tabular}{c | c}
$(C_5)$ &  
\begin{minipage}{0.80\textwidth}
$X,Y$ are complete and first countable space, $Z_0$ is a barreled space,\\ 
$\Phi$ is $K$-epi closed and $0_Z\in\icr (\pi(\dom \Phi))$, 
\end{minipage}
\end{tabular}

\begin{tabular}{c | c}
$(C_6)$ &  
\begin{minipage}{0.7\textwidth}
  $\Phi(x,0_Z)\leqq_K \Phi(x,z)$,   $\forall (x,z)\in \dom\Phi$, \\{\rm and}    $\inte (\pi(\dom \Phi))\ne\emptyset$, 
\end{minipage}
\end{tabular}

\begin{tabular}{c | c}
$(C_7)$ &  
\begin{minipage}{0.7\textwidth}
$0_Z\in \pi(\dom \Phi) - \inte S$ \\    
  $\Phi(x,0_Z)\leqq_K \Phi(x,z)$ whenever  $ (x,z)\in \dom\Phi$ and  $z\geqq_S 0_Z$.
\end{minipage}
\end{tabular}



\begin{remark}\label{rem41} It is worth noticing that  the conditions $(C_1)$-$(C_5)$ extend the ones  proposed in \cite[Theorem 2.7.1]{Z02}  for  the case where  $Y=\mathbb{R}$.    
     Concretely, when $Y=\mathbb{R}$ and $\L(X,Y)$ {replaced by $\{0_{X^*}\}$} the condition $(C_1)$ collapses to  the  condition  \cite[Theorem 2.7.1(i)]{Z02},  the conditions $(C_3)$  and $(C_4)$  are nothing else but the conditions \cite[Theorem 2.7.1(iii)]{Z02} and  \cite[Theorem 2.7.1(viii)]{Z02}, respectively. The conditions    $(C_2)$ and $(C_5)$ 
       generalize \cite[Theorem 2.7.1(ii)]{Z02} and   \cite[Theorem 2.7.1(vi)]{Z02}   to vector problems, respectively. 
\end{remark}


\begin{theorem}[$1^{st}$ Representation of $\epi \Phi(.,0_Z)^*$]
\label{thm_nonasymptotic_representing_epi_2} 
 If $(C_1)$ holds  then 
\begin{equation}
\label{eq_18d}
{\epi \Phi(.,0_Z)^*=\mathcal{M}.}
\end{equation}
If,   in addition that,   $(C_0)$ holds,
 then 
\begin{align}\label{eq_18ddlbis}
{\epi \Phi(.,0_Z)^*=\mathcal{M}
=\mathcal{M}_+. }
\end{align}
\end{theorem}


\begin{proof}
Taking  Proposition \ref{pro_incluepi} into account, to prove \eqref{eq_18d}, it suffices to show  that  
\begin{equation} \label{21nw}
{\epi \Phi(.,0_Z)^*\subset \mathcal{M}}. \end{equation}
For this, take {$(\bar L,\bar y)\in \epi\Phi(.,0_Z)^*$}. Then   by   \eqref{epiF*}, 
\begin{equation}
\label{eq_31bbb} 
	\bar y\notin \bar L(x)-\Phi(x,0_Z)-\inte K,\quad\forall x\in X.
\end{equation} 
 We will show that there exists $\bar T\in\L_\Phi$ such that $(\bar L,\bar y)\in \epi \Phi^*(., \bar T)$. 
   Set
$$
	\Delta_{\bar L}:=\bigcup_{(x,z)\in \dom \Phi}
		\Bigg( \Big(\bar L(x)-\Phi(x,z)-K\Big)\times \{z\} \Bigg).
$$
It is clear that $ \Delta_{\bar L}  \subset Y \times Z_0$ and  it is easy to  check that, as  $\Phi$ is convex,    $\Delta_{\bar L} $  is a convex subset of  $Y \times Z_0$. The proof of \eqref{21nw}   is arranged  in three  steps: 

 {\it Step 1.}  We firstly  prove that    $\inte_{Y\times Z_0}\Delta_{\bar L}\ne\emptyset$ and that  $(\bar y,0_{Z})\notin \inte_{Y\times Z_0} \Delta _{\bar L}$.

$\bullet$  {\it  The proof of  $\inte_{Y\times Z_0}\Delta_{\bar L} \ne\emptyset$.} Indeed, according to  $(C_1)$, there exist $y_{\bar L}\in Y$ and the open neighborhood $V_{\bar L}$ of $0_Z$ such that 
\begin{align}
&\forall z\in V_{\bar L} \cap Z_0,\; \exists x\in X: \Phi(x,z)-\bar L(x)\leqq_K  y_{\bar L} \notag\\
\Longrightarrow\; & \forall z\in V_{\bar L}\cap Z_0,\; \exists x\in X: -y_{\bar L}\in {\bar L}(x)-\Phi(x,z)-K\notag\\
\Longrightarrow\; & \forall (z,k')\in (V_{\bar L}\cap Z_0)\times \inte K,\; \exists x\in X: -y_{\bar L}-k'\in {\bar L}(x)-\Phi(x,z)-K\notag\\
\Longrightarrow\; & \forall (z,k')\in (V_{\bar L}\cap Z_0)\times \inte K,\; (-y_{\bar L} -k',z)\in \Delta_{\bar L}\notag\\
\Longrightarrow\; & (-y_{\bar L}-\inte K )\times (V_{\bar L}\cap Z_0)\subset \Delta_{\bar L}\notag\\
\Longrightarrow\; & (-y_{\bar L}-\inte K )\times (V_{\bar L}\cap Z_0)\subset \inte_{Y\times Z_0}\Delta_{\bar L}.\label{10d}
\end{align}
 So $\inte_{Y\times Z_0}\Delta_{\bar L}\ne\emptyset$ (as $(V_{\bar L}\cap Z_0)\ni 0_Z$ and $\inte K\ne\emptyset$).

$\bullet$ {\it  The proof of  $(\bar y,0_{Z})\notin \inte_{Y\times Z_0} \Delta _{\bar L}$. }
 Assume on  the contrary that $(\bar y,0_{Z})\in \inte_{Y\times Z_0}\Delta _{\bar L}$. 
Then, there exists  a neighborhood $U\times V$ of $(0_{Y},0_{Z})$ such that $[(\bar y+U)\times V]\cap (Y\times Z_0)\subset\Delta _{\bar L}$.  If we take $\bar{k}\in U\cap \inte K$ then  one gets $(\bar y+\bar{k},0_{Z})\in \Delta _{\bar L}$. 
This yields the existence of $(\bar x, 0_Z) \in \dom \Phi$  such that $\bar y+\bar{k}\in {\bar L}(\bar{x})-\Phi(\bar x, 0_Z)-K$,  leading to 
$\bar y\in \bar L(\bar x)-\Phi(\bar x,0_Z)-\inte K,$
 which contradicts \eqref{eq_31bbb}. Thus, $(\bar y,0_{Z})\notin \inte_{Y\times Z_0} \Delta _{\bar L}$.

 {\it Step 2.}    As  $(\bar y,0_{Z})\notin \inte_{Y\times Z_0} \Delta _{\bar L}$, apply   the convex separation theorem (\cite[Theorem 3.4]{Rudin91})  to the point (singleton) $(\bar y,0_{Z})$  and the convex set $\Delta_{\bar L}$ in the space $Y \times Z_0$,  one gets $(y^*_0, z^*_0)\in Y^{*}\times Z^*_0$ such that
\begin{equation}  \label{eq_336d}
	y^*_0(\bar y) < y^*_0(y) + z^*_0(z),\quad \forall (y,z)\in \inte_{Y\times Z_0}\Delta_{\bar L},
\end{equation}
and consequently (as $\Delta_{\bar L}\subset Y\times Z_0$),  
\begin{equation}  \label{eq_337d}
	y^*_0(\bar y) \le y^*_0(y) + z^*_0(z),\quad \forall (y,z)\in \Delta_{\bar L}.
\end{equation}
Next, we  show that
\begin{equation}  \label{eq_338d}
y^*_0(k^\prime)<0, \quad \forall k^\prime\in \inte K.
\end{equation}
Indeed, take $k^\prime  \in\inte K$.  With $y_{\bar L} \in Y$ (exists by $(C_1)$, used in Step 1) and  $k^\prime,  \bar y\in Y$,   by 
   Proposition \ref{pro_1a}(i),  there is  $\mu>0$ such that  $\bar y-\mu k^\prime \in -y_{\bar L}-\inte K$. Hence,  
 $(\bar y-\mu k^\prime ,0_Z)\in (-y_{\bar L} -\inte K)\times (V_{\bar L}\cap Z_0) $ which, by   \eqref{10d}, ensures that $(\bar y-\mu k^\prime ,0_Z)\in  \inte_{Y\times Z_0}\Delta_L$. 
 In turn, 
 \eqref{eq_336d}  leads to  $y^*_0(\bar y)< y^*_0(\bar y-\mu k^\prime )+z^*_0(0_Z)$, or $y^*_0(k^\prime )<0 $, and    \eqref{eq_338d} holds.

 {\it Step 3}.  We now build  an operator $ \bar T\in\L_\Phi$ such that $(\bar L,\bar y)\in \epi \Phi^*(., \bar T)$.

$\bullet$ Take $k_0\in \inte K$ such that $y^\ast_0(k_0)=-1$ (it is possible by \eqref{eq_338d}) 
 and $T\colon Z_0\to Y$ defined by $T(z)=-z^{\ast }_0(z) k_{0}$ for all $z\in Z_0$ ($z_0^*$ exists by the separation theorem in Step 2). It is easy to see that $ T\in\L(Z_0,Y)$ and   for all $z \in Z_0$, it holds
$\big(y_0^{\ast }\circ T\big)(z) =  y_0^*( -z_0^*(z)k_0 )  =  -y_0^* (k_0) z_0^*(z)     =    z^{\ast }_0(z).$
Thus,  \eqref{eq_337d} can be rewritten as
$y^*_0(\bar y) \le y^*_0(y) +(y^*_0\circ T)(z)$ for all $(y,z)\in \Delta_{\bar L},$
or equivalently, $y^{\ast }_0( y+T( z)-\bar y) \geq 0$ for all $(y,z)\in \Delta_{\bar L}.$
So, by   \eqref{eq_338d},   $y+T( z)-\bar y \not\in  \inte K $, yielding  
\begin{equation}  \label{abcd22}
	\bar y\notin y+T(z)-\inte K,\quad \forall (y,z)\in \Delta_{\bar L}.
\end{equation}
Now, as   $\big( \bar L(x)-\Phi(x,z),z\big) \in \Delta_{\bar L} $ for all $(x,z)\in \dom \Phi$, it follows from \eqref{abcd22} that 
\begin{equation}\label{eq_33abc}
	\bar y\notin \bar L(x)-\Phi(x,z) +T(z)-\inte K, \ \forall (x, z) \in \dom\Phi. 
\end{equation}

$\bullet$  On the other hand, using the  Hahn-Banach theorem \cite[Theorem 3.6]{Rudin91}, we can extend  $ z_0^*\in Z_0^\ast $ to $\bar z^*\in Z^*$, and hence, $T$ can be extended to $\bar T =  -\bar z^{\ast }(z) k_{0}         \in\L(Z,Y)$. Consequently,   \eqref{eq_33abc} holds (for all $(x,z)\in \dom \Phi$)
with $T$  being replaced by   $\bar T$. By \eqref{epiF*},   $(\bar L,\bar T,\bar y)\in \epi\Phi^*$,  showing that $(\bar L,\bar T)\in \dom \Phi^\ast$ (or, equivalently, $\bar T\in \L_\Phi$) and 
 {$(\bar L,\bar y)\in    \epi \Phi^*(., \bar T)$}.  
 Thus,   \eqref{21nw} holds and 
 \eqref{eq_18d}  does,  too.


We now    prove that   if,  in addition, $(C_0)$ holds then \eqref{eq_18ddlbis} holds.  For this, we will   show that  
 {under this extra assumption, }   the operator $\bar T$ appeared  in Step 3 for which   \eqref{eq_33abc} holds, must satisfy  $ \bar T \in \L_+(S,K)$,   or equivalently,  $ \bar T(S) \subset K$.   
  We claim firstly  that $ z^*_0 \in -S^+$. For any   $s\in S$, $\nu > 0$, one has 
 $-\nu s\in -S$ and hence,
by $(C_0)$, 
 $\Phi(\tilde x,-\nu s)\leqq_K \Phi (\tilde x,0_Z)$, which implies  that  $(\tilde x, -\nu s) \in \dom \Phi$ 
and that 
  $\bar L(\tilde x)-\Phi(\tilde x,0_Z)\in \bar L(\tilde x)-\Phi(\tilde x, -\nu s)-K$. 
Then $(\bar L(\tilde x)-\Phi(\tilde x,0_Z), - \nu s)\in \Delta_{\bar L}$. Again, by  \eqref{eq_337d}, one has,  for any $\nu > 0$, one has 
$y^*_0(\bar y)\le y^*_0(\bar L(\tilde x)-\Phi(\tilde x,0_Z))+z^*_0(-\nu s),$
or, equivalently, 
$$\frac{1}{\nu}\, y^*_0(\bar y)\le \frac{1}{\nu}\left[y^*_0(\bar L(\tilde x)-\Phi(\tilde x,0_Z))\right] - z^*_0(s), \ \forall \nu > 0.  $$
Letting $\nu$ tends to $+\infty$, one obtains $z^*_0(s)\le 0$. So, $z^\ast_0\in -S^+$. We then have, for all $s \in S$,  
$\bar T(s) = T(s) = - z_0^*(s) k_0 \in K$, showing that $\bar T (S) \subset K$. Consequently, $\bar T \in \L_\Phi^+$, as desired. 
\end{proof}


The next two theorems,   Theorems  \ref{cor_nonasymptotic_representing_epi_3} and \ref{thm_nonasymptotic_representing_epi},  propose other representations   of  $\epi_K \Phi(.,0_Z)^*$
 under different   regularity conditions $(C_2)-(C_7)$,  whose proofs are in the same vein as that of Theorem   \ref{thm_nonasymptotic_representing_epi_2}, and   are left to the  Appendix.

\begin{theorem}[$2^{nd}$   Representation of $\epi \Phi(.,0_Z)^*$]
\label{cor_nonasymptotic_representing_epi_3} 
Let $Z_0$ be the set defined in \eqref{Z0}.
Assume that  at least one of  the  conditions $(C_2)$, $(C_3)$, $(C_4)$, $(C_5)$, or $(C_6)$  holds.
 Then \eqref{eq_18d} holds.
Moreover, if,  in addition that $(C_{0})$ holds,    then \eqref{eq_18ddlbis} holds.
\end{theorem}
\begin{proof}    (See Appendix \ref{A-B}). \end{proof}

\begin{theorem}[$3^{rd}$ Representation of $\epi \Phi(.,0_Z)^*$]
\label{thm_nonasymptotic_representing_epi} 
Assume that $\inte S\ne\emptyset$ and that $(C_7)$ holds. 
 Then,  \eqref{eq_18ddlbis} holds.
 \end{theorem}
\begin{proof}    (See  Appendix \ref{A-C}).\end{proof}



\subsection{General Vector Inequalities - Vector Farkas Lemmas}

We now use the representations of  the epigraph  $\epi  \Phi (.,0_Z)^*$   established in the previous subsection  to  derive characterizations (i.e., conditions that is equivalent to) of the general ``vector inequality"  
  \begin{equation}\label{vector-inequality}
\Phi(x,0_Z) - L(x) + y    \notin -\inte K,\; \forall x\in X, 
\end{equation}
for some $L\in \L(X, Y)$ and $y \in Y$. 
 Each pair of such  equivalent conditions is called an extended Farkas lemma for vector mappings, for instance, the pairs  $(\alpha) $ - $(\beta)$, $(\alpha) $ - $(\gamma)$ in Theorem    \ref{thm_PFL1}  {below} are two  versions of Farkas lemma.   
The same observation applies to  Theorems  \ref{cor_2i}, \ref{cor_3i}.  In some special case with concrete forms mapping $\Phi$  (see Sections 5, 6),  these results  go back to the vector Farkas lemmas  in \cite{DGLL17,DGLMJOTA16,DL2017}  and even  when $Y=\RR$, different systems, different functions $\Phi$ will lead to variants of  results that cover/extend the known extended Farkas lemmas in the literature (see, e.g.,  \cite{Bot2010,DNV-08,DVV-14},  and   references therein).  We say that a version of Farkas lemma is {\it $\V$-stable}  (or we have a  {\it $\V$-stable Farkas lemma})  if   the  pair is  equivalent for all $L \in \V \subset \L(X, Y)$ and $y \in Y$.  
It is called {\it stable} if it is $\mathcal{V}$-stable with $\mathcal{V} = \L(X, Y)$.     We are now in a position to prove   principles for the  $\V$-stability of  vector inequalities \eqref{vector-inequality}. {However, for the simplicity, we state and prove here results with only stable Farkas lemmas (i.e., with 
 with $\mathcal{V} = \L(X, Y)$). The general $\V$-stable vector Farkas lemmas are mentioned in Remark \ref{rem3end}  below. }

\begin{theorem}[Principles for stable vector Farkas lemmas]
\label{thm_PFL1}
 Consider the  statements:  

${\rm (a)}$ {$\epi  \Phi (.,0_Z)^*=\mathcal{M}, $}

$\rm(b)$ {$\epi \Phi (.,0_Z)^*=\mathcal{M}_+, $}

$\rm(c)$ For all {$(L,y)\in \L(X,Y)\times Y$}, two following assertions are equivalent

\hskip1cm $(\alpha)$ $\Phi(x,0_Z)-L(x)+y\notin -\inte K,\; \forall x\in X,$

\hskip1cm $(\beta)$ There exists {$T\in \L_\Phi $} such that 
 $$\Phi(x,z)-L(x)-T(z)+y\notin -\inte K,\quad \forall (x,z)\in X\times Z,\ \ \ \ \ \ \ \ $$

$\rm(d)$ For all {$(L,y)\in \L(X,Y)\times Y$}, two following assertions are equivalent

\hskip1cm $(\alpha)$ $\Phi(x,0_Z)-L(x)+y\notin -\inte K,\; \forall x\in X,$

\hskip1cm $(\gamma)$ There exists {$T\in \L_\Phi^+$} such that 
 $$\Phi(x,z)-L(x)-T(z)+y\notin -\inte K,\quad \forall (x,z)\in X\times Z. \ \ \ \ \ \ \ \ \ $$
Then $[\rm(a)\!   \Leftrightarrow\! \rm(c)]$ and  $[\rm(b)\!   \Leftrightarrow \!\rm(d)]$. 
\end{theorem}

\begin{proof}  Take  {$(L,y)\in \L(X,Y)\times Y$}. Then by 
\eqref{epiF*},  one has 
\begin{eqnarray*}
(\alpha) \  &\Longleftrightarrow&  \ (L,y)\in \epi \Phi(.,0_Z)^*   \ \ \textrm{ and } \\
(\beta)  \  &\Longleftrightarrow& \   \Big(  \exists    T \in \L_\Phi  \ \textrm{s.t.}\   (L,T,y)\in\epi \Phi^*  \Big)\\
&\Longleftrightarrow& \   \Big(  \exists    T \in \L_\Phi  \ \textrm{s.t.}\   (L,y)  \in\epi \Phi^* (\cdot, T) \Big). 
\end{eqnarray*} 
 The first  equivalence  $[\rm(a)\!   \Leftrightarrow\! \rm(c)]$  thus follows.  The proof of  $[\rm(b)\!   \Leftrightarrow\! \rm(d)]$  is similar. 
 \end{proof}

\begin{remark}($\V$-stable vector Farkas lemmas) \label{rem3end} Let $\V \subset \L(X, Y)$ and let 

$({\rm a}^\prime)$  $\epi  \Phi (.,0_Z)^*\cap \V =\mathcal{M} \cap \V, $

$({\rm c}^\prime)$ For all {$L \in \V \times Y$}, two following assertions are equivalent

\hskip1cm $(\alpha)$ $\Phi(x,0_Z)-L(x)+y\notin -\inte K,\; \forall x\in X,$

\hskip1cm $(\beta)$ There exists {$T\in \L_\Phi $} such that 
 $$\Phi(x,z)-L(x)-T(z)+y\notin -\inte K,\quad \forall (x,z)\in X\times Z.\ \ \ \ \ \ \ \ $$
Then one of  {\it ``$\V$-stable vector Farkas lemma"} states that $[({\rm a}^\prime)    \Leftrightarrow   ({\rm c}^\prime)  ]$.  The $\V$-stability of other vector   Farkas lemmas (as 
$[\rm(b)\!   \Leftrightarrow\! \rm(d)]$ in Theorem \ref{thm_PFL1} and in the next theorems) are similar.   
\end{remark}
In the case  when the perturbation mapping is $K$- convex,   regularity conditions $(C_i)$, $i =0, 1, 2, 3, 4, 5, 6, 7$,  ensure  the  $\mathcal{V}$-stability/stability  of  vector inequality       \eqref{vector-inequality}  as shown in the next theorems.   

\begin{theorem}[Stable vector Farkas lemma I]
\label{cor_2i}
Assume that $\Phi$ is a $K$-convex mapping and that the condition $(C_1)$  holds. Then
  $\rm(c)$ in  Theorem \ref{thm_PFL1} holds. If, in addition that   $(C_0)$ fulfills then $\rm(d)$  in  Theorem \ref{thm_PFL1} holds.
\end{theorem}

\begin{proof} When $\Phi$ is a $K$-convex mapping   and $(C_1)$ holds then  by Theorem \ref{thm_nonasymptotic_representing_epi_2}, $\rm(a)$ in Theorem \ref{thm_PFL1} holds and by this theorem      
$\rm (c)$  holds.   If furthermore, $(C_0)$ holds, then again,  Theorem \ref{thm_nonasymptotic_representing_epi_2}  
assures that $\rm(b)$  in Theorems \ref{thm_PFL1} holds,   and by this very theorem, (d) holds. 
\end{proof}

\begin{theorem}[Stable  vector Farkas lemma II]
\label{cor_3i}
Assume that $\Phi$ is a $K$-convex mapping. The following assertions hold:

  $\mathrm{(i)}$\  If $(C_7)$ holds then both  $\rm(c)$ and $\rm (d)$ of  Theorems \ref{thm_PFL1} hold.
  
    $\mathrm{(ii)}$\  If at least one of  the  conditions $(C_2)$, $(C_3), \ldots,  (C_6)$ holds   then   $\rm(c)$ of  Theorem \ref{thm_PFL1} holds. If, in addition that $(C_0)$ holds then   $\rm(d)$ holds as well.
\end{theorem}

\begin{proof} Similar to the proof of Theorem \ref{cor_2i}, $\rm(i)$ follows from  Theorems   \ref{thm_nonasymptotic_representing_epi}  and   \ref{thm_PFL1}    while $\rm(ii)$ follows from  Theorems  \ref{cor_nonasymptotic_representing_epi_3} and    \ref{thm_PFL1}. 
\end{proof}     


\section{Duality for  General Vector Optimization Problems}

Let $X,Y,Z$  and  $K, S$   (with $\inte K \ne \emptyset$) be the spaces and cones as in Section 3, $\pi$ be the canonical projection from $X\times Z$  to $Z$, and     
  $\Phi\colon X\times Z\to Y^\infty$ be a perturbation mapping.  Assume that   
 $ 0_Z \in \pi (\dom \Phi)$\footnote{which  means $\dom \Phi(.,0_Z)\ne\emptyset$, or in the other words, (P) is feasible.}.
Consider  the \emph{general vector optimization problem:} 
\begin{align*}
({\rm P})\qquad \mathop{\winf}\limits_{x\in X}\Phi(x,0_Z).
\end{align*}
The \emph{dual problem} and the \emph{{loose}  dual problem} of $({\rm P})$ are defined respectively as follows:
\begin{align}
 ({\rm D})&\qquad \mathop{\wsup}\limits_{{T\in \L_\Phi}}[-\Phi^* (0_\L,T)], \label{VD} \\  
{({\rm D}_\ell)}&\qquad \mathop{\wsup}\limits_{{T\in \L_\Phi^+}}[-\Phi^* (0_\L,T)],  \label{VDl}
\end{align}
where $\L_\Phi$ and  $\L_\Phi^+$ are the sets defined in \eqref{Lphi} and \eqref{Lphi+}, respectively.

$\bullet$  An operator {$\widehat T\in \L_\Phi$} is said to be  a {\it solution } of  the problem 
$({\rm D})$ if
$$[-\Phi^*(0_\L,\widehat T)]\cap \wsup ({\rm D})\ne \emptyset.$$
\indent  $\bullet$  We say that ``strong duality holds for the pair $({\rm P})-({\rm D})$'' if 
$$\winf ({\rm P})=\wmax ({\rm D}).$$
 It is worth noting that when $\winf ({\rm P})=\wmax ({\rm D})$, one has  $\wsup ({\rm D})=\wmax ({\rm D})$ (see {Proposition \ref{pro_5hh} (i)}), and hence, $({\rm D})$ attains  at any value from $\wsup ({\rm D})$. 

$\bullet$ We denote by $({\rm P}^L)$  the problem $({\rm P})$ perturbed by a linear operator $L\in \L(X,Y)$.  Then  $({\rm P}^L)$  and     
its corresponding   dual and loose  dual problems,  $({\rm D}^L)$ and   $({\rm D}_\ell^L)$ are\footnote{{Note that $\mathop{\wsup}\limits_{{T\in \L_\Phi }}[-(\Phi-L)^* (0_\L,T)] \ \ =\mathop{\wsup}\limits_{{T\in \L_\Phi}}[-\Phi^* (L,T)]$.}}:
\begin{align*}
({\rm P}^L)&\qquad \mathop{\winf}\limits_{x\in X}[\Phi(x,0_Z)-L(x)],\\
({\rm D}^L)&\qquad\mathop{\wsup}\limits_{{T\in \L_\Phi}}[-\Phi^* (L,T)],  \\
({\rm D}_\ell^L)&\qquad\mathop{\wsup}\limits_{{T\in \L_\Phi^+}}[-\Phi^* (L,T)]. 
\end{align*}
\indent $\bullet$  Let $\emptyset \ne \V\subset \L(X,Y)$. It is said that {\it $\V$-stable strong duality holds for the pair  $({\rm P})-({\rm D})$} if the strong duality holds for the pair $({\rm P}^L)-({\rm D}^L)$ for all $L\in \V$. 
 when $\V = \L(X,Y)$, we simply said that  stable strong  duality holds for the pair  $({\rm P})-({\rm D})$.  When $\V = \{0_\L\}$,  $\V$-stable  strong duality is  none other than the usual   strong duality. However, as in Section 3, for the sake of simplicity, we consider in details  only the stable strong duality 
 (i.e., when $\V = \L(X,Y)$) and strong duality (when $\V = \{ 0_{\L} \}$). Statements   on $\V$-stable strong duality are in the similar way as in  
 Remark \ref{rem3end}. 

$\bullet$ The notions of solutions of  the loose   dual problem   $({\rm D}_\ell^L)$,      and   $\V$-stable  strong duality  for  the pair $({\rm P})-({\rm D}_\ell)$  are  defined in the same  way.

 \begin{example} Consider the vector optimization problem
$$({\rm P}^1)\quad \winf \{(x,x^2+2x): x\le 0\}. $$
Let $X = Z = \RR$, $Y = \RR^2$, $S=\mathbb{R}_+$ and $K=\mathbb{R}^2_+$,  and set the perturbation mapping $\Phi\colon \RR\times \RR\to \RR^2$ defined by 
$$\Phi(x,z)=\begin{cases}
(x,x^2+2x+z),&\textrm{if } 2x+z\le 0, \\
+\infty_{\mathbb{R}^2},& \textrm{otherwise}.
\end{cases}$$
It is clear that $\winf ({\rm P}^1)=\mathop{\winf}\limits_{x\in \mathbb{R}}\Phi(x,0)$, and as   
 $\L(X,Y)\cong \mathbb{R}^2$ ,  $\L(Z,Y)\cong \mathbb{R}^2$,   $\L_+(S,K)\cong \mathbb{R}^2_+$,     we have $\Phi^* : \RR^2 \times \RR^2 \rightrightarrows \RR^2$ 
and with $L:=(a,b)\in \RR^2$,  $T:=(c,d)\in  \mathbb{R}^2$, it is easy to see that
\begin{align*}
\Phi^*(L,T)=
&=\wsup[\{((a\!-\!2c\!-\!1)x,(b\!-\!2d)x-x^2): x\in \mathbb{R}\}+\{(cu,(d\!-\!1)u): u\le 0\}]. 
\label{neweqa}
\end{align*}
The last equality leads to (see Proposition \ref{pro_decomp} (iii))
$$\Phi^\ast(L,T)=\wsup\Big[\big\{\big((a-2c-1)x,(b-2d)x-x^2\big): x\in \mathbb{R}\} +\wsup  \big\{(cu,(d-1)u):u\le 0 \big\}\Big],  $$
and   the dual  and the loose dual problems for 
$({\rm P}^1)$ read as 
\begin{align*}
({\rm D}^1)\quad &\mathop{\wsup}_{c\ge 0 \textrm{ or }d\ge 1} \winf \{((2c+1)x,2dx+x^2)-(cu,(d-1)u): x\in \mathbb{R}, \; u\le 0\},\\
({\rm D}_\ell^1)\quad &\mathop{\wsup}_{c\ge 0 \textrm{ and }d\ge 0} \winf \{((2c+1)x,2dx+x^2)-(cu,(d-1)u): x\in \mathbb{R}, \; u\le 0\}.
\end{align*}
\end{example}

\begin{theorem}[Weak duality]
 \label{thm_WD}
For any $L\in \L(X,Y)$, $T\in \L(Z,Y)$ it holds:

  $\mathrm{(i)}$\ $-\Phi^* (L,T)  \preccurlyeq_K\winf ({\rm P}^L)$,

 $\mathrm{(ii)}$\ $\wsup ({\rm D}_\ell^L) \preccurlyeq_K   \wsup ({\rm D}^L) \preccurlyeq_K\winf ({\rm P}^L).$
\end{theorem}

\begin{proof} Firstly, let us denote 
\begin{equation}
\label{defsetMN}
M:=\bigcup_{T\in \L_\Phi } \big(-\Phi^* (L,T)\big), \qquad N:=\{\Phi (x,0_Z)-L(x): x\in X\}.
\end{equation}
Then, it is easy to see that $ \wsup ({\rm D}^L)=\wsup M$ and $\winf ({\rm P}^L)=\winf N$.

(i) Take $\bar x\in X$ and $\bar y\in-\Phi^*(L,T)$.
As $-\bar y\in \Phi^\ast(L,T)=\wsup \{L(x)+T(z)-\Phi(x,z): (x,z)\in X\times Z\}$, by   Proposition \ref{pro_decomp} (v)  one gets
$$-\bar y\notin L(x)+T(z)-\Phi(x,z)-\inte K,\quad \forall (x,z)\in X\times Z, $$
and hence, with  $x=\bar x$ and  $z = 0_Z$,  one has  $\bar y\notin \Phi(\bar x,0_Z)-L(\bar x)+\inte K, $ or equivalently, $\Phi(\bar x,0_Z)-L(\bar x) \not<_K \bar y$. So, 
$-\Phi^* (L,T)  \preccurlyeq_KN$.  Consequently, by Propositions \ref{pro_5hh}(v) and  \ref{prop_4gg}(iv), $-\Phi^* (L,T)=\wsup[ -\Phi^* (L,T)]\preccurlyeq_K   \winf N= \winf ({\rm P}^L)$.

(ii) 
 As $\L_\Phi^+\subset \L_\Phi$, it follows from Proposition \ref{prop_4gg}(iii) 
that 
$\wsup ({\rm D}_\ell^L) \preccurlyeq_K\wsup ({\rm D}^L)$.
   For the last inequality, it follows from (i) that $M\preccurlyeq_K\winf N$, and hence, using  Propositions \ref{pro_5hh}(v),   \ref{prop_4gg}(iv) again, one gets $  \wsup ({\rm D}^L)=\wsup M \preccurlyeq_K\winf[\winf ({\rm P}^L)]=\winf ({\rm P}^L)$. 
  \end{proof}

\begin{theorem}[Characterizations  of  {stable} strong duality for (VP)]
\label{thm_9kk}
  Consider the following statements: 

$\mathrm{(a)}$  $\epi \Phi(.,0_Z)^*=\mathcal{M},$

$\mathrm{(b)}$ {$\epi \Phi(.,0_Z)^*= \mathcal{M}_+,$}

 $\mathrm{(e)}$    The {stable} strong duality holds for the pair  $({\rm P})-({\rm D})$, 
 
$\mathrm{(f)}$  The {stable}  strong duality holds for the pair  $({\rm P})-({\rm D}_\ell)$.

Then $[\rm(a) \! \Leftrightarrow \!\rm(e)]$ and $[\rm(b) \!  \Leftrightarrow\! \rm(f)]$. 
\end{theorem}

\begin{proof}$[\rm(a)\!\Rightarrow\!(e)]$ Assume that $\rm(a)$ holds and let {$L\in \L(X,Y)$}. We will show that 
\begin{equation}\label{eq_26f}
\winf ({\rm P}^L)=\wmax ({\rm D}^L).
\end{equation}
Firstly, as $0_Z\in \pi(\dom \Phi)$,   there is $\tilde x\in X$ such that $\Phi(\tilde x,0_Z)\in Y$, and consequently, \\ $\Phi(\tilde x,0_Z)-L(\tilde x) \in Y$, and  so $\winf ({\rm VP}^L)\ne\{+\infty_Y\}$. 

If  $\winf ({\rm P}^L)=\{-\infty_Y\}$  then, by Theorem \ref{thm_WD},   $\wsup ({\rm D}^L)=\{-\infty_Y\}$, and so,  
$-\Phi^*(L,T)=\{-\infty_Y\}$ for  $T\in \L_\Phi$. Consequently,  
$\wmax ({\rm D}^L)=\{-\infty_Y\}=\winf ({\rm P}^L), $
and \eqref{eq_26f} holds. 
  
Assume now that $\winf ({\rm P}^L)\subset Y$ and  we will show that 
\begin{equation} \label{eq-newa}
\winf ({\rm P}^L) \subset  \wmax ({\rm D}^L).
\end{equation}

$\bullet$
 Let $M$ and $N$ be the sets as in \eqref{defsetMN}  (in  the proof of Theorem \ref{thm_WD}).  Then,  it follows from Theorem \ref{thm_WD} (i) that 
\begin{equation} \label{eqthm62a}
M \preccurlyeq_K\winf ({\rm P}^L) =\winf N.  
\end{equation}

$\bullet$ Take $y\in \winf ({\rm P}^L)=\winf N$, by Proposition \ref{pro_decomp}(v)     (also,   Remark \ref{rem_1hh})  
$y\notin \Phi(x,0_Z)-L(x)+\inte K$ for all $x\in X,$
or equivalently,
\begin{equation} \label{eqnewb} 
\Phi(x,0_Z)-L(x)  - y \notin -\inte K,\quad\forall x\in X.\end{equation} 
Now, as $\rm(a)$ holds,  it follows from Theorem \ref{thm_PFL1}, that \eqref{eqnewb} is equivalent to the fact that 
  there is $\bar T\in L(Z,Y)$ such that  $\bar T  \in \L_\Phi$ and
$\Phi(x,z)-L(x)-\bar T(z)  - y \notin -\inte K$ for all $(x,z)\in X\times Z, $
which is equivalent to (by Proposition \ref{pro_decomp}(iv)) 
\begin{equation}
\label{eq_33hh}
-y\notin \big\{ L(x)+\bar T(z)-\Phi(x,z):  (x,z)\in X\times Z   \big\} -\inte K = \Phi^\ast (L, \bar T) - \inte K.
\end{equation}
Now as  $\Phi^\ast (L, \bar T)   \in \P_p(Y)$, it follows from
  Proposition \ref{pro_5hh} (iii)   that  
  $Y=\big(\Phi^\ast (L, \bar T)     -\inte K\big)\cup \big( \Phi^*(L, \bar T)+K\big), $
which, combining with \eqref{eq_33hh},  yields 
\begin{equation}
\label{eq_34hh}
-y\in \Phi^*(L,\bar T)+K.
\end{equation}

On the other hand, as  $M \preccurlyeq_K \winf N$  (see \eqref{eqthm62a})   and $y \in \winf N$, one has  
$y\not<_K  y'$ for all $y'\in -\Phi^*(L,\bar T)$, yielding  $-y\notin \Phi^*(L,\bar T)+\inte K$.  From   \eqref{eq_34hh} and   Proposition \ref{pro_5hh}(ii), 
$$-y\in (\Phi^*(L,\bar T)+K)\setminus (\Phi^*(L,\bar T)+\inte K)=\Phi^*(L,\bar T), $$
and hence,  $y \in - \Phi^*(L,\bar T)    \subset M$.

$\bullet$
We have shown that  if  $y \in \winf ({\rm P}^L)$  then $y \in M$. Again,  as $M \preccurlyeq_K\winf ({\rm P}^L)$ (see \eqref{eqthm62a}) one has   $y\not<_K y^\prime $ for any $y^\prime \in M$.      So, by  definition of weak maximum,  $y\in \wmax M=\wmax ({\rm D}^L)$, and by 
 the arbitrariness of $y  \in \winf ({\rm P}^L)$, \eqref{eq-newa} has been proved and then 
\begin{equation} \label{eqnewc} \winf ({\rm P}^L)\subset \wmax ({\rm D}^L)\subset \wsup ({\rm D}^L).
\end{equation}
\item Now, as  $\winf ({\rm P}^L), \wsup ({\rm D}^L)\in \mathcal{P}_p(Y)^\bullet$,  Proposition \ref{pro_5hh}(i) yields 
$\winf ({\rm P}^L)= \wsup ({\rm D}^L),$ and we get from  \eqref{eqnewc} that  $\winf ({\rm P}^L)= \wmax ({\rm D}^L),$ which is $\rm (e)$.

$[\rm(e)\!\Rightarrow\!(a)]$ Assume that $\rm(e)$ holds, i.e., \eqref{eq_26f} holds for all {$L\in \L(X,Y)$}. We will show that $\rm(a)$ holds. For this, taking   Proposition \ref{pro_incluepi} into account, it suffices  to show that
\begin{equation}\label{eq_35iii}
{\epi \Phi(.,0_Z)^*\subset\mathcal{M}.}
\end{equation}

Take {$(L,y)\in \epi \Phi(.,0_Z)^*$}. Then   by \eqref{epiF*},  
$\Phi(x,0_Z)-L(x)   + y \notin -\inte K$ for all  $x\in X,$
or equivalently,
$-y\notin N+\inte K$ (where   $N$ is defined in \eqref{defsetMN}),  which yields $\winf N\ne \{-\infty_Y\}$ (see Proposition \ref{pro_decomp}(i) and Remark \ref{rem_1hh}).   
So, $\winf N\subset Y$, and hence, one gets  from   Proposition \ref{pro_decomp}(vi) (see also  Remark \ref{rem_1hh}) that 
$Y= (N+\inte K)\cup \winf N\cup (\winf N-\inte K).    $
Hence (as  $-y\notin N+\inte K$), 
 \begin{equation}\label{eq_35ii}
-y\in \winf N\cup (\winf N-\inte K)\subset \winf N-K=\winf ({\rm VP}^L)-K.
\end{equation}
Now, as $\rm(e)$ holds, for the  given $L$, one has 
$\winf ({\rm P}^L)=\wmax ({\rm D}^L)=\wmax M\subset M, $
(with   $M$ as in \eqref{defsetMN})
which,  together with \eqref{eq_35ii}, yields the existence of $\bar T\in \L_\Phi$ such that $-y\in -\Phi^*(L,\bar T)-K$. The last equation means that $(L,y)\in \epi \Phi^*(.,\bar T)$, and  hence, 
$(L, y) \in  \bigcup_{T\in \L_\Phi }\epi \Phi^*(.,T)  $  and  \eqref{eq_35iii} holds. 
The equivalence  $[\rm (a)\!  \Leftrightarrow\! \rm (e)]$ has been proved. The proof of  $[\rm(b) \!  \Leftrightarrow\! \rm(f)]$ is similar.  \qed
\end{proof}
Further sufficient conditions for the stable  strong duality for $({\rm P})-({\rm D})$,  $({\rm P})-({\rm D}_\ell)$ are  given.  

\begin{theorem}[{Stable} strong  duality I]
\label{thm_12kk}
Assume that $\Phi$ is a $K$-convex mapping and that the  condition $(C_1)$  holds. Then,  {stable} strong  duality holds for the pair $({\rm P})-({\rm D})$.
Moreover, if, in addition that   $(C_{0})$ holds,   then  {stable} strong duality holds for $({\rm P})-({\rm D}_\ell)$ as well.
\end{theorem}

\begin{proof} When  $\Phi$ is a $K$-convex mapping and   the  condition $(C_1)$  holds then  by
Theorem \ref{thm_nonasymptotic_representing_epi_2}, the condition (a) from  Theorem 
\ref{thm_9kk} holds, and hence, by this theorem  the {stable} strong  duality holds for the pair $({\rm P})-({\rm D})$.  If, in addition that   $(C_{0})$ holds  then (b) of Theorem  \ref{thm_9kk} holds, and again, by this theorem the   {stable} strong duality holds for the pair $({\rm P})-({\rm D}_\ell)$. 
\end{proof}

\begin{theorem}[Stable strong  duality II]
\label{thm_14kk}  
Assume that $\Phi$ is a $K$-convex mapping. The following assertions hold:

  $\mathrm{(i)}$\  If $(C_7)$ holds then stable strong  duality holds for the  pairs $({\rm P})-({\rm D})$ and  $({\rm P})-({\rm D}_\ell)$,  
  
  $\mathrm{(ii)}$\  If at least one of  the  conditions $(C_2)$, $(C_3)$, $\ldots$,  $(C_6)$ holds   then    stable strong  duality holds for  $({\rm P})-({\rm D})$. If, in addition that $(C_0)$ holds then   stable strong  duality holds for  $({\rm P})-({\rm D}_\ell)$ as well.
 \end{theorem}
\begin{proof}  Similar to the proof of  Theorem \ref{thm_12kk}, using Theorems   \ref{cor_nonasymptotic_representing_epi_3}, \ref{thm_nonasymptotic_representing_epi}   
and Theorem  \ref{thm_9kk}.  
\end{proof}


\section{Composite Vector  Problems: Lagrange and Fenchel-Lagrange duality} 

In this section we will apply general principles/results  on vector inequalities and on  duality in the previous sections  to some special  classes of vector problems.  We consider a  composed vector optimization problem with a cone constraint and a set constraint:  
\begin{align*}
({\rm CCVP})&\qquad \winf \{F(x)+(\kappa \circ H)(x): \quad x\in C,\; G(x)\in-S\},
\end{align*}
where  $X, Y, Z,W$ are locally convex Hausdorff topological vector spaces, $P$ and $S$ are nonempty convex cone in $W$ and $Z$, respectively,  $F\colon X\to Y^\infty$, $\kappa\colon W\to Y^\infty$ (by convention, $\kappa(+\infty_{W})=+\infty_Y$),  $H\colon X\to W^\infty$, 
 and $G\colon X\to Z^\infty$ are proper mappings, $C$ is a
nonempty subset of $X$.  Let  $A:=C\cap G^{-1}(-S)$ and assume along this section that 
\begin{equation}
\label{eq:42_newew}
A\cap \dom F\cap H^{-1}(\dom \kappa)\neq\emptyset.
\end{equation}  

In order to apply the results obtained in the previous sections to establishing duality results for (CCVP), we need to define suitable perturbation mappings $\Phi $ for (CCVP). 
Concretely, we will construct three typical perturbation mappings $\Phi_1, \Phi_2$, and $\Phi_3$, 
which give rise to variant dual problems and duality results for (CCVP). Another perturbation mapping, namely $\Phi_4$, also is introduced in Remark \ref{rem712} together  with  the corresponding dual problems.

\medskip

\textit{\textbf{The First  Perturbation Mapping $\boldsymbol{\Phi_1}$: Lagrange duality. }}  Take $\tilde Z:=W\times Z$ as the space of  perturbation variables and define the perturbation mapping $\Phi_1\colon X\times \tilde Z\to Y^\infty$, 
\begin{equation}
\label{phi1}
\Phi_1(x,w,z)=\begin{cases} F(x)+\kappa(H(x)+w), & \textrm{if } x\in C \textrm{ and } G(x)+z\in -S, \\
+\infty_Y,& \textrm{otherwise}.
\end{cases}
\end{equation}
We consider the cone $\tilde S:=P\times S$ in $\tilde Z$. Observe that  $\L_+(\tilde S,K)\cong\L_+(P,K)\times \L_+(S,K)$.

Firstly, one has $\Phi_1(x,0_{W}, 0_{Z})=F(x)+(\kappa \circ H)(x)+I_A(x)$ for all $x\in X$,  and 
 \begin{gather*} 
\dom \Phi_1=\left\{(x,w,z)\in X\times W\times Z: 
x\in \tilde C,\; w\in -H(x)+\dom \kappa,\; z\in -G(x)-S
\right\},    
  \end{gather*} 
and so, 
$  \pi_1(\dom \Phi_1)  =-\left\{(H(x),G(x)): x\in \tilde C\right\}+\dom\kappa \times (-S),$ where $\pi_1\colon X\times \tilde Z\to \tilde Z$ defined by $\pi_1(x,w,z)=(w,z)$
and $\tilde C:=C\cap \dom F\cap \dom G\cap\dom H$.
 As \eqref{eq:42_newew} holds,  one has  $(0_{W},0_{Z})\in \pi_1 (\dom \Phi_1)$.

\begin{lemma}
\label{lem:6.1_nwewewE}

 $\mathrm{(i)}$\  For all $L\in \L(X,Y)$ and $T:=(T_1,T_2)\in \L(W,Y)\times \L(Z,Y)$, it holds
\begin{equation}
\label{eq:45nwewE}
\Phi_1^*(L,T)=(F+T_1\circ H+T_2\circ G+I_{C})^*(L)\uplus\kappa^\ast(T_1)\uplus I^*_{-S}(T_2).
\end{equation}
Moreover,  if $T_2\in \L_+(S,K)$ then
$\Phi_1^*(L,T)=(F+T_1\circ H_1+T_2\circ H_2+I_{C})^*(L)\uplus\kappa^\ast(T_1), $\\ 
\indent  $\mathrm{(ii)}$\  It holds:
\begin{align*}
\bigcup_{T\in {\L_{\Phi_1}}}\epi \Phi^*_1(.,T)
&=\bigcup_{\substack{T_1\in \dom \kappa^\ast\\ T_2\in \L_+^w(S,K)}}
\Psi(\exepi (F+T_1\!\circ\! H+T_2\!\circ\! G+I_{C})^\ast\boxplus (0_\L, \kappa^\ast(T_1)\uplus I_{-S}^\ast(T_2)),\\
\bigcup_{T\in {\L_{\Phi_1}^+}}\epi \Phi^*_1(.,T)
&=\bigcup_{\substack{T_1\in \dom \kappa^\ast\cap\L_+(P,K)\\ T_2\in \L_+(S,K)}}
\Psi(\exepi (F+T_1\!\circ\! H+T_2\!\circ\! G+I_{C})^\ast\boxplus (0_\L, \kappa^\ast(T_1)).
\end{align*}
\end{lemma}

\begin{proof} (See  the Apendix D).     \qed\end{proof} 

According to  Lemma \ref{lem:6.1_nwewewE}(ii), the sets $\mathcal{M}$ and $\mathcal{M}_+$ in \eqref{eq:14_nwewewE} now becomes,   respectively,  
\begin{align}
\mathcal{M}^1&:=\bigcup_{\substack{T_1\in \dom \kappa^\ast\\ T_2\in \L_+^w(S,K)}}
\Psi(\exepi (F+T_1\!\circ\! H+T_2\!\circ\! G+I_{C})^\ast\boxplus (0_\L, \kappa^\ast(T_1)\uplus I_{-S}^\ast(T_2)),\\
\mathcal{M}^1_+&:=\bigcup_{\substack{T_1\in \dom \kappa^\ast\cap\L_+(P,K)\\ T_2\in \L_+(S,K)}}
\Psi(\exepi (F+T_1\!\circ\! H+T_2\!\circ\! G+I_{C})^\ast\boxplus (0_\L, \kappa^\ast(T_1)).
\end{align}
By Lemma \ref{lem:6.1_nwewewE}(i), 
for $ T:=(T_1,T_2)\in \L(Z_1,Y)\times \L(Z_2,Y)$, 
\begin{align*} \Phi_1^*(0_\L,T)&= (F+T_1\circ H+T_2\circ G+I_{C})^*(0_\L)\uplus \kappa^\ast(T_1)\uplus I^*_{-S}(T_2).
\end{align*}
 If  $ T_2\in \L_+(S_2,K) $  then 
$\Phi_1^*(0_L,T)=(F+T_1\circ H+T_2\circ G+I_{C})^*(0_\L)\uplus \kappa^\ast(T_1).$ 
 Note that 
\begin{align*}
&-(F+T_1\circ H+T_2\circ G+I_{C})^*(0_\L)\uplus \kappa^\ast(T_1)\uplus I^*_{-S}(T_2)\\
&=-\wsup[(F+T_1\circ H+T_2\circ G+I_{C})^*(0_\L)+ \kappa^\ast(T_1)+ I^*_{-S}(T_2)] \\
&=-\wsup[\wsup(-F-T_1\circ H-T_2\circ G-I_{C})(X)+ \wsup(T_1-\kappa)(Z)+ \wsup(-T_2)(-S)] \\
&=\mathop{\winf}\limits_{(x,z,s)\in C\times Z\times S}[F(x)+(T_1\circ H)(x)+(T_2\circ G)(x)+\kappa(z)-T_1(z)+ T_2(s)], 
\end{align*}
and so,  the
{\it  dual problem}  and the {\it loose  dual problem} of $({\rm CCVP})$ which are specific cases of $({\rm D})$ and $({\rm D}_\ell)$   with $\Phi=\Phi_1$,  become the  {\it Lagrange dual} and {\it loose Lagrange dual problems}:  
\begin{eqnarray*}
\label{eq_42kkk}
({\rm CCVD}^1) && \hskip1.3em\mathop{\wsup}_{\substack{T_1\in\dom \kappa^\ast\\T_2 \in\L_+^w(S,K)}} \hskip1.3em
 \mathop{\winf}\limits_{(x,z,s)\in C\times Z\times S}[F(x)\!+\!(T_1\!\circ\! H)(x)\!+\!(T_2\!\circ\! G)(x)\!+\!\kappa(z)\!-\!T_1(z)\!+\! T_2(s)],  \\
\label{eq_42kk}
({\rm CCVD}_\ell^1) &&   \mathop{\wsup}_{\substack{T_1\in\dom \kappa^\ast\cap \L_+(P,K)\\T_2 \in\L_+(S,K)}} \mathop{\winf}\limits_{(x,z)\in C\times Z}[F(x)\!+\!(T_1\!\circ\! H)(x)\!+\!(T_2\!\circ\! G)(x)\!+\!\kappa(z)\!-\!T_1(z)], 
\end{eqnarray*}

When $\kappa \equiv 0$, $({\rm CCVP}) $ collapses to the vector optimization problem $({\rm VP}) $ in Section 6.2 (see also \cite{CDLP20}) and then the above two dual problems   collapse to the ``familiar" Lagrange and loose Lagrange  dual problems $({\rm VD}^1)$ and $({\rm VD}_\ell^1)$ in  Section 6.2.

\begin{corollary}[Principle of stable strong duality for (CCVP)]\label{cor_7.8zzz}
Consider the  statements:  

$({\rm a}_1)$  $ \epi (F+\kappa\circ H+I_A)^\ast = \mathcal{M}^1$,

$({\rm b}_1)$ $\epi (F+\kappa\circ H+I_A)^\ast =  \mathcal{M}^1_+$, 

$({\rm e}_1)$  The  {stable} strong duality holds for the pair  $({\rm CCVP})-({\rm CCVD}^1)$, 

$({\rm f}_1)$  The  {stable} strong duality holds for the pair  $({\rm CCVP})-({\rm CCVD}_\ell^1)$.

\noindent
Then, it holds $[({\rm a}_1)\!\Leftrightarrow\! ({\rm e}_1)]$ and $[({\rm b}_1 )\!  \Leftrightarrow\! ({\rm f}_1)]$.
\end{corollary}

\begin{proof}
In the case $\Phi=\Phi_1$, we have shown that $\Phi(.,0_{W}, 0_{Z})=F+(\kappa \circ H)+I_A$, $\M=\M^1$ and $\M_+=\M^1_+$ while the dual problems dual problems  $({\rm CCVD})$ and $({\rm CCVD}_\ell)$ are of the form of  $({\rm D})$ and $({\rm D}_\ell)$,  respectively. The conclusion now follows from Theorem \ref{thm_9kk}.  
\end{proof}

\begin{corollary}[Stable strong Lagrange duality for (CCVP)]\label{corol86}
Assume that $F$ is $K$-convex, that $H$ is $P$-convex, that $G$ is $S$-convex,  that  $\kappa$ is $K$-convex and $(P,K)$-nondecreasing, and  that $C$ is convex. Denote $Z_0^1:=\lin [\{(H(x),G(x)): x\in C\}+(-\dom\kappa )\times S]$
 and assume $(C_1^1)$  holds:

\begin{tabular}{c | c}
$(C_1^1)$ &  
\begin{minipage}{0.8\textwidth}
$\forall {L\in \L(X,Y)}$, $\exists y_L\in Y$,  $\exists V_L\in\mathcal{N}(0_{Z_1},0_{Z_2})$ s.t.
$$\forall (w,z)\in  V_L\cap Z_0^1,\; \exists x\in C: G(x)+z\in -S\textrm{ and } F(x)+\kappa(H(x)+w)-L(x)\leqq_K y_L.$$
\end{minipage}
\end{tabular}

\noindent Then,   {stable} strong duality holds for the  pairs $({\rm CCVP})-({\rm CCVD}^1)$ and  $({\rm CCVP})-({\rm CCVD}_\ell^1)$.
\end{corollary}

\begin{proof}
Assumptions on convexity of the mappings $F,H,G,\kappa$ and the set $C$ and on the non-decreasing property of $\kappa$ entail that $\Phi_1$ is $K$-convex mapping. The condition $(C_1^1)$ ensures that $(C_1)$ in previous sections holds with $\Phi=\Phi_1$. So, according to  Theorem \ref{thm_12kk},  the stable strong duality holds for the  pairs $({\rm CCVP})-({\rm CCVD}^1)$.

Moreover, as $A\cap\dom F\cap H^{-1}(\dom \kappa)\ne\emptyset$, there is $\tilde x\in X$ such that $\tilde x\in C$ and $G(\tilde x)\in -S$ which yields $\Phi_1(\tilde x,0_W,0_Z)=F(\tilde x)+(\kappa\circ H)(\tilde x)$. Take $w,z\in -(P\times S)$. As $z\in -S$, $G(\tilde x)+z\in -S$ as well, and hence, $\Phi_1(\tilde x, w,z)=F(\tilde x)+\kappa(H(\tilde x)+w)$. This, together with the fact that $w\in -P$ and that $\kappa$ is $(P,K)$ non-decreasing, yields $\Phi_1(\tilde x, w,z)\leqq_KF(\tilde x)+(\kappa\circ H)(\tilde x)= \Phi_1(\tilde x,0_W,0_Z)$. We have just proved that $(C_0)$ holds with $\Phi=\Phi_1$. So, again by Theorem \ref{thm_12kk},  the {stable} strong duality holds for the  pairs $({\rm CCVP})-({\rm CCVD}^1_\ell)$.   
\end{proof}
\begin{remark} \label{rem71}
Observe  that one can specify conditions $(C_2)-(C_7)$ to this setting and   use Theorem \ref{thm_14kk}      (instead of  Theorem \ref{thm_12kk})     to     get stable strong duality results for pairs $({\rm CCVP})-({\rm CCVD}^1)$,   $({\rm CCVP})-({\rm CCVD}_\ell^1)$. However, nto to make the paper becomes too long, we omit them. 
Similar observation also applies to other perturbation mappings $\Phi_2$ and $\Phi_3$ below.   
\end{remark}
 
With $\Phi=\Phi_1$,   $\V$-stability of general vector inequalities in Theorem \ref{thm_PFL1} leads to
\begin{corollary}[Stable Farkas lemma for constrained composite vector systems 1]
\label{cor:6.3www}
 Consider the  following statements:  

$({\rm a}_1)$ {$ \epi (F+\kappa\circ H+I_A)^\ast = \mathcal{M}^1$},

$({\rm b}_1)$  {$\epi (F+\kappa\circ H+I_A)^\ast = \mathcal{M}^1_+$}, 

$({\rm c}_1)$  For all {$(L,y)\in \L(X,Y)\times Y$}, two following assertions are equivalent

\hskip1cm $(\alpha')$ $x\in C,\;  G(x)\in -S\; \Longrightarrow\;\; F(x)+(\kappa\circ H)(x)-L(x)+y\notin -\inte K.$

\hskip1cm $(\beta')$ $\exists T_1\in \dom \kappa^\ast,\; T_2\in \L_+^w(S,K)$ such that 
 $$ \ \ \ \ \ \ \ \ \ F(x)+(T_1\circ H)(x)+(T_2\circ G)(x)-L(x)+y\notin \kappa^\ast(T_1)+T_2(-S)-\inte K, \forall x\in C, $$

$({\rm d}_1)$ For all $(L,y)\in \L(X,Y)\times Y$, two following assertions are equivalent

\hskip1cm $(\alpha')$ $\Phi(x,0_Z)-L(x)+y\notin -\inte K,\; \forall x\in X,$

\hskip1cm $(\gamma')$ $\exists T_1\in \dom \kappa^\ast\cap\L_+(P,K),\; T_2\in \L_+(S,K)$ such that 
 $$F(x)+(T_1\circ H)(x)+(T_2\circ G)(x)-L(x)+y\notin \kappa^\ast(T_1)-\inte K,\quad \forall x\in C.$$

\noindent
Then $[({\rm a}_1)\!   \Leftrightarrow\! ({\rm c}_1)]$ and  $[({\rm b}_1)\! \Leftrightarrow\! ({\rm d}_1)]$. 
\end{corollary}

\textit{\textbf{The second   perturbation mapping $\Phi_2$: Fenchel-Lagrange duality I.}}
Take $\tilde Z := X\times W\times Z$  as  the space of  perturbation variables (with the cone  $\tilde S:=\{0_X\}\times P\times S$, and consequently, $\L_+(\tilde S, K)  \cong \L(X,Y)\times \L_+(P,K)\times \L_+(S,K)$). We define    the perturbation mapping 
$\Phi_2\colon X\times \tilde Z\to Y^\infty$,
\begin{equation}
\label{phi2}
\Phi_2(x,v,w,z)=\begin{cases}F(x+v)+\kappa(H(x)+w), & \textrm{if } x\in C \textrm{ and } G(x)+z\in -S\\
+\infty_Y,& \textrm{otherwise},
\end{cases}
\end{equation}
 We now give shortly some key properties of the perturbation $\Phi_2$ which can be proved by the similar arguments as the ones for the first perturbation $\Phi_1$ (see Lemma \ref{lem:6.1_nwewewE}):

$\bullet$ It holds that $\Phi_2(x,0_X,0_{W}, 0_{Z})=F(x)+(\kappa \circ H)(x)+I_A(x)$ for all $x\in X$,  and  
$$  \pi(\dom \Phi_2)  =-\Big\{(x,H(x),G(x)): x\in C\cap \dom H\cap\dom G\Big\}+\dom F\times \dom \kappa \times (-S).$$

$\bullet$ For all $L\in \L(X,Y)$ and $T:=(L',T_1,T_2)\in  \L(X,Y)\times \L(W,Y)\times \L(Z,Y)$, one has
$$\Phi_2^*(L,T)=F^\ast(L')\uplus(T_1\circ H+T_2\circ G+I_C)^*(L-L')\uplus\kappa^\ast(T_1)\uplus I^*_{-S}(T_2),$$
and if $T_2\in \L_+(S,K)$ then 
$$\Phi_2^*(L,T)=F^\ast(L')\uplus(T_1\circ H+T_2\circ G+I_C)^*(L-L')\uplus\kappa^\ast(T_1).$$

$\bullet$ The sets $\mathcal{M}$ and $\mathcal{M}_+$ (defined by \eqref{eq:14_nwewewE}) in this case   become, respectively
\begin{align}  
\mathcal{M}^2&:=\bigcup_{\substack{T_1\in \dom \kappa^\ast\\ T_2\in \L_+^w(S,K)}}
\Psi\left(\!\exepi F^\ast\! \boxplus \!\exepi(T_1\!\circ\! H\!+\!T_2\!\circ\! G\!+\!I_C)^\ast\!\boxplus\! (0_\L\!, \kappa^\ast\!(T_1)\uplus I_{-\!S}^\ast\!(T_2)\right),\label{eq:A2}\\
\mathcal{M}^2_+&:=\bigcup_{\substack{T_1\in \dom \kappa^\ast\cap\L_+(P,K)\\ T_2\in \L_+(S,K)}}
\Psi\left(\exepi F^\ast  \boxplus \exepi(T_1\!\circ\! H\!+\!T_2\!\circ\! G\!+\!I_C)^\ast\boxplus (0_\L, \kappa^\ast(T_1)\right). \label{eq:A2l}
\end{align}

$\bullet$ In this  case ($\Phi=\Phi_2$,  $\tilde Z$ plays the role of $Z$),   the dual problems $({\rm D})$ and $({\rm D}_\ell)$  lead to  different {Fenchel-Lagrange  dual problems} of  $({\rm CCVP})$ as follows:
\begin{align*}
({\rm CCVD}^2)&\quad \hskip1.3em\mathop{\wsup}_{\substack{L'\in \L(X,Y)\\T_1\in \dom \kappa^\ast\\T_2\in \L_+^w(S,K)}} \hskip1.3em
  \big[-F^\ast(L')\uplus (T_1\!\circ\! H\!+\!T_2\!\circ\! G\!+\!I_C)^*(\!-\!L'\!)\uplus \kappa^\ast(T_1) \uplus I^*_{-S}(T_2)\big],\\
({\rm CCVD}_\ell^2)&\quad \mathop{\wsup}_{\substack{L'\in \L(X,Y)\\T_1\in \dom \kappa^\ast\cap \L_+(P,K)\\T_2\in \L_+(S,K)}}   
\!\big[-F^\ast(L')\uplus (T_1\circ H+T_2\circ G+I_C)^*(-L')\uplus \kappa^\ast(T_1)\big].
\end{align*}

The next corollary is a  consequence of Theorems \ref{thm_9kk}. 

\begin{corollary}[Principle of stable strong duality for (CCVP)]\label{cor_7.8zzz2}
Consider the  statements:

$({\rm a}_2)$ { $\epi (F+\kappa\circ H+I_A)^\ast =\mathcal{M}^2$},

$({\rm b}_2)$ { $\epi (F+\kappa\circ H+I_A)^\ast  =\mathcal{M}^2_+$},

$({\rm e}_2)$  The {stable} strong duality holds for the pair  $({\rm CCVP})-({\rm CCVD}^2)$, 

$({\rm f}_2)$  The  {stable} strong duality holds for the pair  $({\rm CCVP})-({\rm CCVD}_\ell^2)$.

\noindent Then, it holds $[({\rm a}_2)\!\Leftrightarrow\! ({\rm e}_2)]$ and $[({\rm b}_2 ) \! \Leftrightarrow \!({\rm f}_2)]$.
\end{corollary}

Under some  assumptions  that guarantee $\Phi_2$ to be convex,  and under one of the  conditions $(C_0),(C_1), (C_2), \cdots, (C_6)$\footnote{As the cone  $S=\{0_X\}\times S_1\times S_2$ has empty interior,   $(C_7)$ fails.}  specified to the case  $\Phi =\Phi_2$,  we will get stable duality for pairs $({\rm CCVP})-({\rm CCVD}^2)$ and $({\rm CCVP})-({\rm CCVD}_\ell^2)$. 
However,  we  consider here (and for the case $\Phi_3$ below as well)  only the results concerning the assumption   $(C_1)$. 
Denote 
$$Z_0^2:=\lin \Big[\{(x,H_1(x),H_2(x)): x\in C\cap \dom H_1\cap\dom H_2\}-\dom F\times \dom \kappa \times (-S_2)\Big].$$

\begin{corollary}[Stable strong Fenchel-Lagrange duality for (CCVP)]\label{corol72}
Assume that $F$ is $K$-convex, that $H$ is $P$-convex, that $G$ is $S$-convex,  that  $\kappa$ is $K$-convex and $(P,K)$ non-decreasing, and  that $C$ is convex. 
 Assume further that  the following condition holds:
 
\noindent
\begin{tabular}{c | c}
\!\!\!$(C_1^2)$\!\! &  \!\!\!
\begin{minipage}{0.85\textwidth}
$\forall {L\in \L(X,Y)}$, $\exists y_L\in Y$,  $\exists W_L   \in  \mathcal{N}(0_X,0_{Z_1},0_{Z_2})$  such that 
$$\forall (x',w,z)\in  W_L\cap Z_0^2,\; \exists x\in C: G(x)+z\in -S,\; F(x+x')+\kappa (H(x)+w)-L(x)\leqq_K y_L.$$
\end{minipage}
\end{tabular}

\noindent Then, the {stable} strong duality holds for pairs $({\rm CCVP})-({\rm CCVD}^2)$ and $({\rm CCVP})-({\rm CCVD}_\ell^2)$.
\end{corollary}

\begin{proof}   The proof is similar to that of  Corollary \ref{corol86}. 
\end{proof}

A version of vector Farkas-lemma  specialized from Theorem \ref{thm_PFL1} when  $\Phi=\Phi_2$ reads as:
\begin{corollary}[Stable Farkas lemma for constrained composite vector systems 2]
\label{cor:6.6www}
 Consider the  following statements:  

$({\rm a}_2)$ {$\epi (F+\kappa\circ H+I_A)^\ast = \mathcal{M}^2$},

$({\rm b}_2)$  {$\epi (F+\kappa\circ H+I_A)^\ast =  \mathcal{M}^2_+$}, 

$({\rm c}_2)$ For all {$(L,y)\in \L(X,Y)\times Y$}, two following assertions are equivalent

\hskip1cm
$(\alpha')$ $x\in C,\;  G(x)\in -S\; \Longrightarrow\;\; F(x)+(\kappa\circ H)(x)-L(x)+y\notin -\inte K.$

\hskip1cm $(\delta')$ $\exists L'\in \L(X,Y),\; T_1\in \dom \kappa^\ast,\; T_2\in \L_+^w(S,K)$ such that 
 $$F(x)-L'(x)+(T_1\circ H)(x')+(T_2\circ G)(x')-(L-L')(x')+y\notin \kappa^\ast(T_1)+T_2(-S)-\inte K$$ 
for all $(x,x')\in X\times C.$

$({\rm d}_2)$ For all {$(L,y)\in \L(X,Y)\times Y$}, two following assertions are equivalent

\hskip1cm $(\alpha')$ $\Phi(x,0_Z)-L(x)+y\notin -\inte K,\; \forall x\in X,$

\hskip1cm $(\epsilon')$ $\exists L'\in \L(X,Y),\; \exists T_1\in \dom \kappa^\ast\cap\L_+(P,K),\; T_2\in \L_+(S,K)$ such that 
$$F(x)-L'(x)+(T_1\circ H)(x')+(T_2\circ G)(x')-(L-L')(x')+y\notin \kappa^\ast(T_1)-\inte K$$ 
\null\hskip1cm for all $(x,x')\in X\times C.$

\noindent Then $[({\rm a}_2)\!   \Leftrightarrow \!({\rm c}_2)]$ and  $[({\rm b}_2)\! \Leftrightarrow\! ({\rm d}_2)]$. 
\end{corollary}


\textit{\textbf{ The third   perturbation mapping $\Phi_3$: Fenchel-Lagrange duality II. }} 
Let    $\tilde Z:= X\times X\times W\times Z$ be   the space of  perturbation variables and define the  perturbation mapping $\Phi_3\colon X\times \tilde Z\to Y^\infty$,
\begin{equation}
\label{phi3}
\Phi_3(x,x',x''\!,w,z)=\begin{cases} F(x+x')+\kappa (H(x)+w), & \textrm{if } x+x''\in C,  G(x)+z\in -S\\
+\infty_Y,& \textrm{otherwise}.
\end{cases}
\end{equation}
We consider in $\tilde Z$ the cone $\tilde S:= \{0_X\}\times \{0_X\}\times P\times S$. Then  $\L_+(S,K)\cong \L(X,Y)\times \L(X,Y)\times \L_+(P,K)\times \L_+(S,K)$. We make some quick  observations:

$\bullet$ It can be checked that $\Phi_3(x,0_X,0_X,0_{W}, 0_{Z})=F(x)+(\kappa \circ H)(x)+I_A(x)$ for all $x\in X$ and 
$ \pi(\dom \Phi_3)  =-\Big\{(x,x,H(x),G(x)):x\in \dom H\cap\dom G\Big\}+\dom F\times C\times \dom \kappa  \times (-S).$

$\bullet$ For all $L\in \L(X,Y)$ and $T:=(L',L'',T_1,T_2)\in (\L(X,Y))^2\times \L(W,Y)\times \L(Z,Y)$, it holds
$\Phi_3^*(L,T)=F^\ast(L')\uplus(T_1\!\circ\! H+T_2\!\circ\! G)^*(L-L'-L'')\uplus I_C^\ast(L'')\uplus \kappa^\ast(T_1)\uplus I^*_{-S}(T_2),$
and if $T_2\in \L_+(S,K)$ then 
$\Phi_3^*(L,T)=F^\ast(L')\uplus(T_1\circ H+T_2\circ G)^*(L-L'-L'')
\uplus I_C^\ast(L'')\uplus\kappa^\ast(T_1).$

$\bullet$ The sets $\mathcal{M}$ and $\mathcal{M}_+$ (when  $\Phi=\Phi_3$) become, respectively
\begin{align*}
\mathcal{M}^3&:=
\bigcup_{\substack{T_1\in \dom \kappa^\ast\\ T_2\in \L_+^w(S,K)}}
\Psi(\exepi F^\ast \boxplus \exepi(T_1\!\circ\! H+T_2\!\circ\! G)^\ast\boxplus \exepi I_C^\ast\boxplus (0_\L, \kappa^\ast\!(T_1)\uplus I_{-\!S}^\ast\!(T_2)),
\\
\mathcal{M}^3_+&:=
\bigcup_{\substack{T_1\in \dom \kappa^\ast\cap\L_+(P,K)\\ T_2\in \L_+(S,K)}}
\Psi(\exepi F^\ast \boxplus \exepi(T_1\circ H+T_2\circ G)^\ast\boxplus \exepi I_C^\ast\boxplus (0_\L, \kappa^\ast(T_1)).
\end{align*}

$\bullet$ We now get two forms of  {Fenchel-Lagrange  dual problems} of $({\rm CCVP})$ by specializing $({\rm D})$ and $({\rm D}_\ell)$ with $\Phi=\Phi_3$:
\begin{align*}
({\rm CCVD}^3)&\; \hskip1.3em\mathop{\wsup}_{\substack{L',L''\in \L(X,Y)\\T_1\in \dom \kappa^\ast\\T_2\in \L_+^w(S,K)}} 
  \left[
-F^\ast(L')\uplus (T_1\!\circ\! H+T_2\!\circ\! G)^*(-\!L'\!-\!L'')\uplus I_C^\ast(L'')\uplus \kappa^\ast\!(T_1)\uplus I^*_{-\!S}\!(T_2)
\right],\\
({\rm CCVD}_\ell^3)&\; \mathop{\wsup}_{\substack{L',L''\in \L(X,Y)\\T_1\in \dom \kappa^\ast\cap \L_+(P,K)\\T_2\in \L_+(S,K)}}\hskip-1.3em
   \big[-F^\ast(L')\uplus (T_1\circ H+T_2\circ G)^*(-L'-L'')\uplus I_C^\ast(L'')\uplus \kappa^\ast(T_1)\big].
\end{align*}
Theorems  \ref{thm_9kk} and \ref{thm_12kk} applying to   $\Phi_3$ lead to  characterizations of  stable   strong  {Fenchel-Lagrange}  duality for  $({\rm CCVP})$:  

\begin{corollary}[Principle of stable strong duality for (CCVP)]\label{cor_7.8zzz2b}
Consider the following statements:

$({\rm a}_3)$ {$\epi (F+\kappa\circ H+I_A)^\ast =\mathcal{M}^3$},

$({\rm b}_3)$ {$\epi (F+\kappa\circ H+I_A)^\ast  =\mathcal{M}^3_+$},

$({\rm e}_3)$  The {stable} strong duality holds for the pair  $({\rm CCVP})-({\rm CCVD}^3)$, 

$({\rm f}_3)$  The {stable} strong duality holds for the pair  $({\rm CCVP})-({\rm CCVD}_\ell^3)$.

\noindent Then, it holds $[({\rm a}_3)\!\Leftrightarrow\! ({\rm e}_3)]$ and $[({\rm b}_3 )\!  \Leftrightarrow\! ({\rm f}_3)]$.
\end{corollary}

Denote 
\begin{multline*}
Z_0^3:=\lin \Big[\{(x,x,H(x),G(x)):x\in \dom H\cap\dom G\}-\dom F\times C\times \dom \kappa  \times (-S)\Big].
\end{multline*}

\begin{corollary}[Stable strong duality for (CCVP)]
\label{corol73}
Assume that $F$ is $K$-convex, that $H$ is $P$-convex, that $G$ is $S$-convex,  that  $\kappa$ is $K$-convex and $(P,K)$ non-decreasing, and  that $C$ is convex. 
 Assume further that  the following condition holds:

\begin{tabular}{c | c}
$( C_1^3)$&  
\begin{minipage}{0.84\textwidth}
$\forall L\in \L(X,Y)$, $\exists y_L\in Y$,  $\exists W_L \in \mathcal{N} (0_X,0_X,0_{W},0_{Z})$  such that \\
$\forall (x'\!,x''\!, w,z)\in  W_L\!\cap \!Z_0^3,\; \exists x\!\in\! X: x\!+\!x''\!\in\! C, \ G(x)\!+\!z\in -S, \\ 
\phantom{mm} F(x+x')+ \kappa (H(x)+w)\!-\!L(x)\!\leqq_K\! y_L.$
\end{minipage}
\end{tabular}

\noindent Then, {stable} strong duality holds for the pairs $({\rm CCVP})-({\rm CCVD}^3)$ and  $({\rm CCVP})-({\rm CCVD}_\ell^3)$.
\end{corollary}

Theorem \ref{thm_PFL1} with   $\Phi=\Phi_3$,   leads to: 
\begin{corollary}[Stable Farkas lemma for constrained composite vector systems 3]
\label{cor:6.9www}
 Consider the  following statements:  

$({\rm a}_3)$ {$\epi (F+\kappa\circ H+I_A)^\ast =\mathcal{M}^3$},

$({\rm b}_3)$ {$\epi (F+\kappa\circ H+I_A)^\ast = \mathcal{M}^3_+$}, 

$({\rm c}_3)$ For all {$(L,y)\in \L(X,Y)\times Y$}, two following assertions are equivalent

\hskip1cm $(\alpha')$ $x\in C,\;  G(x)\in -S\; \Longrightarrow\;\; F(x)+(\kappa\circ H)(x)-L(x)+y\notin -\inte K.$

\hskip1cm $(\zeta')$ $\exists L',L''\in \L(X,Y),\; T_1\in \dom \kappa^\ast,\; T_2\in \L_+^w(S,K)$ such that 
 $$F(x)-L'(x)+(T_1\circ H)(x')+(T_2\circ G)(x')-(L-L'-L'')(x')-L''(x'')+y\notin \kappa^\ast(T_1)+T_2(-S)-\inte K$$ 
for all $(x,x',x'')\in X\times X\times C.$

$({\rm d}_3)$ For all {$(L,y)\in \L(X,Y)\times Y$}, two following assertions are equivalent

\hskip1cm $(\alpha')$ $\Phi(x,0_Z)-L(x)+y\notin -\inte K,\; \forall x\in X,$

\hskip1cm $(\eta')$ $\exists L',L''\in \L(X,Y),\; \exists T_1\in \dom \kappa^\ast\cap\L_+(P,K),\; T_2\in \L_+(S,K)$ such that 
 $$F(x)-L'(x)+(T_1\circ H)(x')+(T_2\circ G)(x')-(L-L'-L'')(x')-L''(x'')+y\notin \kappa^\ast(T_1)-\inte K$$ 
for all $(x,x',x'')\in X\times X\times C.$

	\noindent Then $[({\rm a}_3)  \! \Leftrightarrow\! ({\rm c}_3)]$ and  $[({\rm b}_3)\! \Leftrightarrow\! ({\rm d}_3)]$. 
\end{corollary}


\begin{remark}\label{rem712}
It worth noting that there may have  other ways to define perturbation mappings for (CCVP), and then, more representations of $\epi (F+\kappa\circ H+I_A)^\ast $,  more corresponding results on duality for (CCVP),  vector Farkas lemmas can be derived. For instance, take $\tilde Z := X^3\times W\times Z$,    $\tilde S:=\{0_X\}^3\times P\times S$,   
$\Phi_4\colon X^4\times W\times Z\to Y^\infty$, 
\begin{equation*}
\Phi_4(x,x'\!,x''\!,x'''\!,w,z)=\begin{cases} F(x\!+\!x')+\kappa (H\!_1(x\!+\!x'')\!+\!w), & \textrm{if } x\!+\!x'''\!\in \!C,   G(x)\!+\!z\!\in\! -\!S, \\
+\infty_Y,& \textrm{otherwise}.
\end{cases}
\end{equation*}
This  perturbation mapping $\Phi_4 $  leads to new types of dual problems of (CCVP): 
\begin{align*}
({\rm CCVD}^4)&\; \hskip1.3em\mathop{\wsup}_{\substack{L',L'',L'''\in \L(X,Y)\\T_1\in \dom \kappa^\ast\\T_2\in \L_+^w(S,K)}} 
  \left[
\begin{aligned}
-F^\ast(L')\uplus (T_1\!\circ\! H)^\ast(L'')\uplus(T_2\!\circ\! G)^*(-\!L'\!-\!L''\!-\!L''')\uplus\\\uplus I_C^\ast(L''')\uplus \kappa^\ast\!(T_1) \uplus I^*_{-\!S}\!(T_2)
\end{aligned}\right],\\
({\rm CCVD}_\ell^4)&\; \mathop{\wsup}_{\substack{L',L'',L'''\in \L(X,Y)\\T_1\in \dom \kappa^\ast\cap \L_+(P,K)\\T_2\in \L_+(S,K)}}\hskip0.3em
    \left[
\begin{aligned}
-F^\ast(L')\uplus (T_1\!\circ\! H)^\ast(L'')\uplus(T_2\!\circ\! G)^*(-\!L'\!-\!L''\!-\!L''')\uplus\\\uplus I_C^\ast(L''')\uplus \kappa^\ast\!(T_1)
\end{aligned}\right],  
\end{align*}
and under suitable conditions, one can get stable strong duality results  for pairs $({\rm CCVP})-({\rm CCVD}^4)$ and $({\rm CCVP})-({\rm CCVD}_\ell^4)$, and some forms of vector Farkas lemmas as well.  
\end{remark}

\section{Special cases of (CCVP)}


\subsection{Lagrange and Fenchel-Lagrange Duality for Composite Vector Problems}

We are concerned with the 
{\it composite  vector problem}
\cite{CDLP20}
\begin{align*}
({\rm CVP})&\qquad \winf \{F(x)+(\kappa \circ H)(x):x\in X\}
\end{align*}
where  $X, Y, W$,  $P$, $F$,  $\kappa$, $H$  are  as in Section 5. Assume that $\dom F\cap H^{-1}(\dom \kappa)\ne \emptyset$. It is clear that 
 (CVP) is a special case of (CCVP), and  
 $\L(Z,Y)=\L_+(S,K)=\L(Y,Y)$ and $I^\ast_{-S}(T)=-\bd K$ for all $T\in \L(Y,Y)$.
The dual  problems $({\rm CCVD}^1)$ and $({\rm CCVD}^1_\ell)$ now become
\begin{eqnarray*}
({\rm CVD}^1) && \hskip2em\mathop{\wsup}_{\substack{T\in\dom \kappa^\ast}} \hskip2em
 \mathop{\winf}\limits_{(x,z)\in X\times Z} [F(x)+T\circ H(x) +\kappa(z) - T(z)],  \\
({\rm CVD}_\ell^1) &&   \mathop{\wsup}_{\substack{T\in\dom \kappa^\ast\cap \L_+(P,K)}} \mathop{\winf}\limits_{(x,z)\in X\times Z} [F(x)+T\circ H(x) +\kappa(z) - T(z)], 
\end{eqnarray*}
which are  nothing else but the Lagrange and loose Lagrange dual problems $(\rm CVD)$, $({\rm CVD}_l)$ introduced recently in  \cite{CDLP20} while the problems $({\rm CCVD}^2)$ (and $({\rm CCVD}^3))$  and $({\rm CCVD}^2_\ell)$  (and $({\rm CCVD}^3_\ell)$) lead to new {\it Fenchel-Lagrange  dual problems} of $({\rm CVP})$: 
\begin{align*}
({\rm CVD}^2)&\quad \hskip2em\mathop{\wsup}_{\substack{L'\in \L(X,Y)\\T\in \dom \kappa^\ast}} \hskip2em
  \big[-F^\ast(L')\uplus (T\circ H)^*(-L')\uplus \kappa^\ast(T) \big],\\
({\rm CVD}_\ell^2)&\quad \mathop{\wsup}_{\substack{L'\in \L(X,Y)\\T\in \dom \kappa^\ast\cap \L_+(P,K)}}   \big[-F^\ast(L')\uplus (T\circ H)^*(-L')\uplus \kappa^\ast(T)\big].
\end{align*}
The sets $\mathcal{M}^1$, $\mathcal{M}^1_+$, $\mathcal{M}^2$ (or $\mathcal{M}^3$), and  $\mathcal{M}^2_+$ (or $\mathcal{M}^3_+$) reduces to, respectively
\begin{align*}
\mathcal{A}^1&:=\hskip1.9em\bigcup_{\substack{T\in \dom \kappa^\ast}}\hskip1.9em
\Psi(\exepi (F+T\circ H)^\ast\boxplus (0_\L, \kappa^\ast(T)),\\
\mathcal{A}^1_+&:=\bigcup_{\substack{T\in \dom \kappa^\ast\cap \L_+(P,K)}}
\Psi(\exepi (F+T\circ H)^\ast\boxplus (0_\L, \kappa^\ast(T)),\\
\mathcal{A}^2&:=\hskip1.9em\bigcup_{\substack{T\in \dom \kappa^\ast}}\hskip1.9em
\Psi(\exepi F\boxplus \exepi(T\circ H)^\ast\boxplus (0_\L, \kappa^\ast(T)),\\
\mathcal{A}^2_+&:=\bigcup_{\substack{T\in \dom \kappa^\ast\cap \L_+(P,K)}}
\Psi(\exepi F\boxplus \exepi(T\circ H)^\ast\boxplus (0_\L, \kappa^\ast(T)).
\end{align*}
Observe that $\mathcal{A}^1$ and $\mathcal{A}^1_+$ are respectively the sets  $\A$ and $\B$  in  \cite{CDLP20}.
The next corollary  is a direct consequence of  Corollaries \ref{cor_7.8zzz} and \ref{cor_7.8zzz2} while Corollary \ref{cor:72www} follows directly from Corollaries \ref{cor:6.3www}, \ref{cor:6.6www}.

\begin{corollary}\label{cor:7.1}
Let $i\in\{1,2\}$. 
Consider the following statements:

$({\rm a}^1_i)$ $ \epi (F+\kappa\circ H)^\ast = \mathcal{A}^i$,

$({\rm b}^1_i)$ $\epi (F+\kappa\circ H)^\ast = \mathcal{A}^i_+$, 

$({\rm e}^1_i)$  The stable strong duality holds for the pair  $({\rm CVP})-({\rm CVD}^i)$, 

$({\rm f}^1_i)$  The stable strong duality holds for the pair  $({\rm CVP})-({\rm CVD}_\ell^i)$.

\noindent
Then, for all $i\in\{1,2\}$,  it holds $[({\rm a}_2^i)\Longleftrightarrow ({\rm e}_2^i)]$ and $[({\rm b}_2^i )  \Longleftrightarrow ({\rm f}_2^i)]$.
\end{corollary}

\begin{corollary}[Stable Farkas lemma for systems of involving composite mappings]
\label{cor:72www}
For each $(L,y)\in \L(X,Y)\times Y$, consider the  following assertions:  

$(\alpha^1)$ $F(x)+(\kappa\circ H)(x)-L(x)+y\notin -\inte K,\; \forall x\in X,$

$(\beta^1)$ $\exists T\in \dom \kappa^\ast:F(x)+(T\circ H)(x)-L(x)+y\notin \kappa^\ast(T)-\inte K,\quad \forall x\in C,$

$(\gamma^1)$ $\exists T\in \dom \kappa^\ast\cap\L_+(P,K):F(x)+(T\circ H)(x)-L(x)+y\notin \kappa^\ast(T)-\inte K,\quad \forall x\in C,$

$(\delta^1)$ $\exists L'\in \L(X,Y),\; T\in \dom \kappa^\ast$ such that 
 $$F(x)-L'(x)+(T\circ H)(x') -(L-L')(x')+y\notin \kappa^\ast(T)-\inte K,\; 
\forall (x,x')\in X\times C,$$

$(\epsilon^1)$ $\exists L'\in \L(X,Y)\cap\L_+(P,K),\; T\in \dom \kappa^\ast$ such that 
 $$F(x)-L'(x)+(T\circ H)(x') -(L-L')(x')+y\notin \kappa^\ast(T)-\inte K,\; 
\forall (x,x')\in X\times C.$$
\noindent Then, one has: 

$\rm(i)$\ \ $({\rm a}^1_1) \Longleftrightarrow \big[(\alpha_1)\!\Leftrightarrow\! (\beta^1),\ \forall (L,y)\in \L(X,Y)\times Y\big]$,

$\rm(ii)$\ 
 $({\rm b}^1_1) \Longleftrightarrow \big[(\alpha_1)\!\Leftrightarrow \!(\gamma^1),\ \forall (L,y)\in \L(X,Y)\times Y\big]$, 

$\rm(iii)$\
 $({\rm a}^1_2) \Longleftrightarrow \big[(\alpha_1)\!\Leftrightarrow\! (\delta^1),\ \forall (L,y)\in\L(X,Y)\times Y\big]$,

$\rm(iv)$\ $({\rm b}^1_2) \Longleftrightarrow \big[(\alpha_1)\!\Leftrightarrow \!(\epsilon^1),\ \forall (L,y)\in \L(X,Y)\times Y\big]$, \\
where $({\rm a}^1_i)$, $({\rm b}^1_i)
$, $ i = 1, 2$ are as in Corollary \ref{cor:7.1}.
\end{corollary}

\begin{remark}   $(\mathrm{i})$  The case $ i = 1$ in Corollary \ref{cor:7.1} and  the conclusions   $\rm(i)$-$\rm(ii)$ in Corollary \ref{cor:72www}  go back to   
 \cite[Theorems 3.6, 3.7]{CDLP20} while the results corresponding to $i = 2$ in Corollary  \ref{cor:7.1} and  the conclusions  $\rm(iii)$-$\rm(iv)$ in Corollary \ref{cor:72www}, 
 up to the best knowledge of the author, are new.   \\
\indent $(\mathrm{ii})$   When specifying the regularity conditions $(C_0)$-$(C_7)$) in Section 3 to $(\mathrm{CVP})$ we will get variants of sufficient conditions for $({\rm a}^1_i)$, $({\rm b}^1_i)$, $ i = 1, 2$, which means also 
 sufficient conditions for the stable strong duality for (CVP) in Corollaries \ref{cor:7.1}-\ref{cor:72www}.        
\end{remark}  


\subsection{Lagrange and Fenchel-Lagrange Duality for Cone-Constrained Vector Problems}


We retain  the notations in Section 5 and consider 
 the {\it cone-constrained vector problem}
\begin{align*}
({\rm VP})&\qquad \winf \{F(x):x\in C,\; G(x)\in -S\}. 
\end{align*}
Assume that   $A\cap \dom F\ne \emptyset$, where  $A:=C\cap G^{-1}(-S)$. 
It is clear that (VP) is a special case of (CCVP)  and it was studied  in \cite{DGLL17,DGLMJOTA16,DL20}. Note that in this case, one has   $\L(W,Y)=\L(P,K)=\{\theta\}$ where $\theta\colon \{0\}\to W$,    $\theta(0)=0_W$, and 
hence, $\dom \kappa^\ast=\{\theta\} $. 

The {\it Lagrange, Fenchel-Lagrange dual problems} of $({\rm CCVD}_\ell^1) $, $({\rm CCVD}_\ell^2) $, and $({\rm CCVD}_\ell^3) $  (in Section 5) now turn back to exactly the same as  the Lagrange and Fenchel-Lagrange dual problems  in \cite{DL20}. Corollaries \ref{cor_7.8zzz}, \ref{cor_7.8zzz2}, and \ref {cor_7.8zzz2b} lead to characterizations of stable strong duality for (VP) (similar to Corollary \ref{cor:7.1}) which cover \cite[Theorem 6.1]{DL20} while the similar results as Corollary \ref{cor:72www} cover \cite[Theorem 5.1]{DL20}. 
On the other hand, the problems  $({\rm CCVD}^i)$, $i = 1, 2, 3$,  collapses to the following ones: 
\begin{eqnarray*}
({\rm VD}^1) && \mathop{\wsup}_{T \in\L_+^w(S,K)}   \mathop{\winf}\limits_{(x,s)\in C\times S} [F(x)+T\circ G(x)+T(s)],  \\
({\rm VD}^2) && \mathop{\wsup}_{\substack{L'\in \L(X,Y)\\ T \in\L_+^w(S,K)}}   [-F^\ast(L')\uplus(T\circ G+I_C)^*(-L')\uplus I^*_{-S}(T)],  \\
({\rm VD}^3) && \mathop{\wsup}_{\substack{L',L''\in \L(X,Y)\\ T \in\L_+^w(S,K)}}   [-F^\ast(L')\uplus(T\circ G)(-L'-L'')\uplus I_C^*(-L'')\uplus I^*_{-S}(T)],    \\
\end{eqnarray*}
which are totally new dual problems for (VP),  and so, all the stable duality results (specified from theorems/corollaries in Section 5  to (VP)) concerning these problems are new. 
As far as the length of the paper is concerned,  we will not state these results here. The same applies to the class of cone-constrained composite (scalar)  problems in the next subsection.


 \subsection{Scalar Cone-Constrained Composite Problems} 
 
Consider the {\it cone-constrained composite (scalar) problem} of the form 
\begin{align*}
({\rm CCP})&\qquad \inf \{f(x)+(\kappa \circ H)(x) :x\in C,\; G(x)\in-S\}.
\end{align*}
Here we retain the notions as in Section 5 with  $Y=\mathbb{R}$ and $K=\mathbb{R}_+$.  Obviously, this is a special case of   (CCVP).   This model was considered in \cite{BGW07}  with the case where   $f=0$    or  $G\equiv 0$, $C = X$. 
Noting that, in this case, $\L_+(P,K)=P^+$ and $\L_+^w(S,K)=\L_+(S,K)=S^+$, and so,  $({\rm CCVD}^1)$ and $({\rm CCVD}^1_\ell)$ respectively  become
 the {\it Lagrange dual problems:} 
$$({\rm CCD}^1) \qquad \mathop{\sup}_{(\lambda_1,\lambda_2)\in\dom \kappa^\ast \times S^+}  
 \left[\inf_{x\in C}[f(x)+(\lambda_1 H)(x)+(\lambda_2 G)(x)]  -  \kappa^\ast(\lambda_1)\right],  $$
$$({\rm CCD}^1_\ell) \qquad \mathop{\sup}_{\substack{\lambda_1 \in\dom \kappa^\ast\cap P^+\\ \lambda_2 \in S^+}}
 \left[\inf_{x\in C}[f(x)+ (\lambda_1 H)(x)+(\lambda_2 G)(x)]  -  \kappa^\ast(\lambda_1)\right],  $$
and  $({\rm CCD}^i)$,   $({\rm CCD}^i_\ell)$, $i=2,3$, become  some forms of Fenchel-Lagrange dual problems: 
\begin{align*}
({\rm CCD}^2) \quad&\mathop{\sup}_{\substack{(x^\ast\!, \lambda_1,\lambda_2) \in X^\ast\!\times \dom \kappa^\ast \times S^+}}  
  \left[-f^\ast(x^\ast) -(\lambda_1 H+\lambda_2 G+i_C)^*(-x^\ast) - \kappa^\ast(\lambda_1)\right],\\
({\rm CCD}^2_\ell ) \quad&\hskip0.8em\mathop{\sup}_{\substack{(x^\ast\!,\lambda_2)\in X^\ast\!\times S^+\\\lambda_1 \in\dom \kappa^\ast\cap P^+}}  \hskip0.8em
  \left[-f^\ast(x^\ast) -(\lambda_1 H+\lambda_2 G+i_C)^*(-x^\ast) - \kappa^\ast(\lambda_1)\right],\\
({\rm CCD}^3) \quad&\mathop{\sup}_{\substack{x^\ast\!, y^\ast\in X^\ast\\
 (\lambda_1,\lambda_2) \in  \dom \kappa^\ast \times S^+}}  
  \left[-f^\ast(x^\ast) -(\lambda_1 H+\lambda_2 G)^\ast(-x^\ast-y^\ast)-i_C^*(y^\ast) - \kappa^\ast(\lambda_1)\right],\\
({\rm CCD}^3_\ell) \quad&\mathop{\sup}_{\substack{(x^\ast\!, y^\ast\!,\lambda_2)\in (X^\ast)^2\times S^+\\
 \lambda_1 \in  \dom \kappa^\ast \cap P^+}}  
  \left[-f^\ast(x^\ast) -(\lambda_1 H+ \lambda_2 G)^\ast(-x^\ast-y^\ast)-i_C^*(y^\ast) - \kappa^\ast(\lambda_1)\right].
\end{align*}
Corollaries \ref{cor_7.8zzz}-\ref{corol73} lead to various results on strong duality for (CCP).  Some of the important special cases of the problem (CCP) 
can be listed as: the {\it composite (scalar) problem} (CP)  (see \cite{Bot2010,BGW07})   and the cone-constrained problem (P1): 
\begin{align*}
({\rm CP})&\qquad \inf_{x\in X} [f(x)+(\kappa \circ H)(x)], \\
({\rm P1})&\qquad \inf \{f(x)) :x\in C,\; G(x)\in-S\}.  
\end{align*}
\indent For (CP),  the Lagrange dual problems $({\rm CCD}^1)$ and $({\rm CCD}^1_\ell)$ become the usual Lagrange dual problems  of (CP)  while  
 $({\rm CCD}^2)$   and $({\rm CCD}^2_\ell)$ reduce, respectively, to Fenchel-Lagrange dual problems appeared in \cite[page 42]{Bot2010} and in \cite{BGW07}. 
 
 Similarly, for the (P1), $({\rm CCD}^1)$ and $({\rm CCD}^1_\ell)$ become the usual Lagrange dual problems while  
 the dual problems  $({\rm CCD}^2)$ (and $({\rm CCD}^2_\ell)$) and $({\rm CCD}^3)$ (and $({\rm CCD}^3_\ell)$) become, 
\begin{align*}
({\rm D}^2) \quad&\mathop{\sup}_{\substack{(x^\ast\!, \lambda) \in X^\ast \times S^+}}  
  \left[-f^\ast(x^\ast) -(\lambda G+i_C)^*(-x^\ast) \right],\\
({\rm D}^3) \quad&\mathop{\sup}_{\substack{x^\ast\!, y^\ast\in X^\ast\\
 \lambda \in   S^+}}  
  \left[-f^\ast(x^\ast) -(\lambda G)^\ast(-x^\ast-y^\ast)-i_C^*(y^\ast) \right], 
\end{align*}
respectively,  which are the forms of Fenchel-Lagrange duality introduced in  \cite{Bot2010,DNV-08,BGWMIA09,DVN-08}    
and this  justifies the name ``Fenchel-Lagrange dual problem" for the problems $ ({\rm CCVD}^i)$, $({\rm CCVD}_\ell^i) $, $ i = 2, 3$  in Section 5.
Corollaries \ref{cor_7.8zzz}, \ref{cor_7.8zzz2}, and \ref{cor_7.8zzz2b}, in the current setting, implies \cite[Corollary 6.2]{DL20}, and hence, they cover  the results established in \cite{Bot2010,DNV-08,BGWMIA09,DVN-08}  (see also \cite[Remark 6.2]{DL20}). 

 It is also  worth observing that   when $Y=\mathbb{R}$, $(C_3)$ and $(C_4)$ reduce to $(RC_1^\Phi)$ and $(RC_3^\Phi)$ in \cite{Bot2010} respectively,  while $(C_5)$ is weaker than $(RC_2^\Phi)$ and consequently, Theorem \ref{thm_14kk}, specified to the case $Y=\mathbb{R}$,  extends   \cite[Theorem 1.7]{Bot2010} in  the sense that  we get a stable   strong duality results (instead of strong duality results) and under  weaker assumptions. 
 \medskip 

\noindent
{\small {\bf Acknowledgements}\quad This research is funded  by Vietnam National University HoChiMinh city (VNU-HCM) under grant number B2021-28-03. }


{\small 
\appendix
 
 \section*{Appendix} 
   
   
\section{\small Proof of Theorem \ref{lem_epiclosed}    (Extended open mapping theorem)}

 Consider the set-valued mapping $\R\colon X\times Y \rightrightarrows Z_0$ defined by
$$\R(x,y):= \{z\in Z: \Phi (x,z)\leqq_K y\}\subset Z_0.$$
Then, \eqref{eq_2.8d} simply means that  $0_Z\in \inte_{Z_0} \mathcal{G}(U_0\times V_0)$. The proof of  \eqref{eq_2.8d} is  based on   Lemma  \ref{lem_Zalinescu} and    proceeds   with  four  steps as follows:

  {\rm ($\alpha$)} Let  $U$ and $V$  be the convex neighborhoods of $x_0$ and $\Phi(x_0, 0_Z)$, respectively. 
  It is easy to check that $\R(U\times V)$ is a convex set (using the convexity of $\Phi$ and of the sets $U,V$). 
We will show that  $\R(U\times V)$  is  absorbing  in   $Z_0$.  
 Take $z\in Z_0$, we will show that there exists $\lambda>0$ such that $\lambda z\in \R(U\times V)$. Firstly, one has  $0_Z\in \pi(\dom \Phi)$ (as $(x_0,0_Z)\in \dom \Phi$), so $\aff (\pi(\dom \Phi))=\lin (\pi(\dom \Phi))=Z_0$. 
Now, as $0_Z\in \icr(\pi(\dom \Phi))$, there exists $\delta>0$  such that $\delta z \in \pi(\dom \Phi)$. 
Then, there is $x\in X$ such that $(x,\delta z)\in\dom \Phi$, or equivalently, $\Phi(x,\delta z)\in Y$. 
On the one hand, as $U,V$ are respectively the neighborhoods of $x_0$ and $\Phi(x_0,0_Z)$,  there exists $\mu\in]0;1[$ such that $x_0+\mu (x-x_0)\in U$ and $\Phi (x_0,0_Z)+\mu (\Phi (x,\delta z)-\Phi (x_0,0_Z))\in V$. On the other hand, one also has
\begin{align*}
\Phi (x_0+\mu(x-x_0), \mu \delta z) &=\Phi ((1-\mu)(x_0,0_Z)+\mu(x,\delta z))\\
			&\in  (1-\mu) \Phi (x_0,0_Z) + \mu \Phi (x,\delta z)-K \textrm{ (as $\Phi$ is $K$-convex)}\\
			&=\Phi (x_0,0_Z)+\mu (\Phi (x,\delta z)-\Phi (x_0,0_Z))-K. 
\end{align*}
Thus, $ \lambda z    \in \R(U\times V)$, where $\lambda =  \mu\delta >0$,  meaning that $\R(U\times V)$ is absorbing in $Z_0$.

  {\rm ($\beta$)}  Now, take an arbitrary  neighborhood $U\times V$ of $(x_0, \Phi(x_0, 0_Z))$ and we  will prove that   $0_Z\in \inte_{Z_0}\cl( \R(U\times V))$.      As $X\times Y$ is locally convex, replace $U,V$ by their subsets if necessary, we can suppose that $U$ and $V$ are convex.
From  ($\alpha$), $\R(U\times V)$ is convex and absorbing in $Z_0$. Consequently, $\cl \R(U\times V)$ is a convex, closed and absorbing subset of $Z_0$, and hence, a neighborhood of $0_Z$  (as $Z_0$ is a barreled space). So,  $0_Z\in \inte_{Z_0}\cl( \R(U\times V))$.

{\rm ($\gamma$)} \  We now show  that \eqref{eq_2.8d} follows from Lemma \ref{lem_Zalinescu}.     From  ($\beta$),  $0_Z$ belongs to the intersection of all sets of the forms  $\inte_{Z_0}(\cl( \R(U\times V))$ where   $U\times V$ running from the collection of all neighborhoods of  $(x_0,\Phi (x_0,0_Z))$. Moreover, by assumption,  $X\times Y$ is a complete and first countable space, and $\R$ is a closed convex multifunction (as $\Phi$ is $K$-convex and $K$-epi closed). Then   Lemma \ref{lem_Zalinescu}, applying to the multifunction $\mathcal{G}$ with $X\times Y$, $Z_0$,  and  $(x_0,\Phi (x_0,0_Z)) $  playing the roles of $\tilde X$, $\tilde Y$ and $\tilde x_0$, respectively,  gives  
 $$0_Z \ \in \  \bigcap\limits_{U\times V \in \mathcal{N}(x_0,\Phi (x_0,0_Z))}   \   \inte_{Z_0}(\R(U\times V)), $$ 
 showing that $0_Z \in \inte_{Z_0} \mathcal{G}(U_0\times V_0)$  and   \eqref{eq_2.8d} follows.  \qed


\section{\small Proof of Theorem  \ref{cor_nonasymptotic_representing_epi_3}} \label{A-B} 

We will show that if one of the conditions $(C_2)$, $(C_3)$, $(C_4)$, or $(C_5)$ holds then the condition $(C_1)$ in Theorem \ref{thm_nonasymptotic_representing_epi_2} holds, and hence, the conclusion now follows from Theorem \ref{thm_nonasymptotic_representing_epi_2}.

$(\alpha)$ Assume that $(C_2)$ holds. Then, for any $L\in \L(X,Y)$, take $y_L:=\widehat y-L(\widehat x)$ and $V_L:=\widehat V$, one gets
$\Phi(\widehat x,z)-L(\widehat x)\leqq_K y_L$ for all $z\in V_L\cap Z_0$. 	
So, $(C_1)$ holds (with $\hat x$ playing the role of   $x$).

$(\beta)$ Assume that $(C_3)$ holds.  Pick $\bar k\in \inte K$. Then, $\Phi(\widehat x,0_Z)+\bar k -\inte K$ is a neighborhood of $\Phi(\widehat x, 0_Z)$. So, by the continuity of $\Phi(\widehat x,.)_{| Z_0}$  at $0_Z$, there is a    neighborhood  $\widehat V$ of $0_Z$ such that $\Phi(\widehat x, \widehat V\cap Z_0) \subset \Phi(\widehat x,0_Z)+\bar k-\inte K$, or equivalently, 
$$ \Phi (\widehat x, z)  \leqq_K      \Phi(\widehat x,0_Z)+\bar k:=\widehat y,\quad \forall z\in \widehat V\cap Z_0,$$
showing that $(C_2)$ holds, and consequently,  $(C_1)$ holds as well.

$(\gamma)$ Assume that $(C_4)$ holds. 
It is worth noting that $\aff (\pi (\dom \Phi))=\lin (\pi (\dom \Phi))=Z_0$ (as $0_Z\in \pi(\dom\Phi)$). So, $\ri (\pi (\dom \Phi))=\inte_{Z_0} (\pi (\dom \Phi))$.

Without loss the generality, assume that $Z_0=\mathbb{R}^n$ and take $\{e_i\}_{i=1}^n$ the standard basic of $\mathbb{R}^n$ (i.e., the $i^{th}$ coordinate of $e_i$ is $1$ and the others are $0$). 
As $0_Z\in \inte_{Z_0}  (\pi (\dom \Phi))$, there exists $\{\varepsilon_i\}_{i=1}^n\subset \RR_+\setminus \{0\}$ such that $\{\pm \varepsilon_i e_i\}_{i=1}^n\subset  \pi (\dom \Phi) $. 
For each $i= 1,2,\ldots, n$, as $\pm \varepsilon_i e_i \in \pi (\dom \Phi)$, there exists $x_i,x'_i\in X$ such that $(x_i,\varepsilon_ie_i)$ and $(x'_i, -\varepsilon_i e_i)$ belong to $\dom \Phi$.
Next, for any $L\in\L(X,Y)$, take $y_L\in Y$ such that (the existence of $y_L$ is guaranteed by Lemma \ref{pro_1a}(ii))
\begin{equation}
\label{eq_339d}
\Phi(x_i, \varepsilon_i e_i)-L(x_i)<_K y_L   \ \ {\rm and} \ \, \,\Phi (x'_i, -\varepsilon_ie_i)-L(x_i')<_K y_L,
 \forall i\in\{1,2,\ldots, n\}.
\end{equation}
It is easy to see that $\co\big(\{\pm \varepsilon_i e_i\}_{i=1}^n\big)$ is a neighborhood of $0_Z$ in $Z_0$, and hence, there exists the neighborhood $\widehat V$ of $0_Z$ such that $\widehat V\cap Z_0=\co\big(\{\pm \varepsilon_i e_i\}_{i=1}^n\big)$.
Now, for each $z\in \widehat V\cap Z_0$,  there exists $\{(\lambda_i,\lambda'_i)\}_{i=1}^n\subset \RR^2_+$ such that $\sum_{i=1}^{n}(\lambda_i+\lambda'_i)=1$ and $z=\sum_{i=1}^{n}(\lambda_i \varepsilon_i - \lambda'_i \varepsilon')e_i$. 
Take $\tilde x= \sum_{i=1}^{n} (\lambda_i x_i + \lambda'_ix'_i) $, by the convexity of $\Phi$, 
\begin{equation*}
\Phi (\tilde x, z) \leqq_K \sum_{i=1}^{n} \big(\lambda_i \Phi ( x_i, \varepsilon_ie_i)+ \lambda'_i \Phi ( x'_i, -\varepsilon_ie_i)\big).
\end{equation*}
This, together with \eqref{eq_339d}, entails $\Phi(\tilde x, z)-L(\tilde x )\leqq_K y_L$ which means that 
 $(C_1)$ holds.

$(\delta)$ Assume that $(C_5)$ holds.
Take $L\in \L(X,Y)$. As $0_Z\in \pi(\dom \Phi)$, there exists $x_0\in X$ such that $(x_0,0_Z)\in \dom \Phi$. It is worth noting that as $\Phi $ is $K$-convex and $K$-epi closed, $\Phi - L$ is $K$-convex and $K$-epi closed as well.
Pick $k_0\in \inte K$. Apply the extended open mapping theorem,  Theorem \ref{lem_epiclosed},  with $\Phi$ being replaced by $\Phi - L$, $U_0=X$ and $V_0=\Phi(x_0,0_Z)-L(x_0)+k_0-\inte K$, one gets that $0_Z \in \inte_{Z_0} V_1$, where $V_1$  is  the set defined by 
$$V_1:= \left\{z\in Z: \begin{array}{l}
\Phi (x,z)-L(x) \leqq_K y \textrm{ for some } x\in X \\
\textrm{and } y\in \Phi(x_0,0_Z)-L(x_0)+k_0-\inte K
\end{array}\right\}.$$
Take $y_L=\Phi(x_0,0_Z)-L(x_0)+k_0$ and $V_L$ the neighborhood of $0_Z$ such that $V_L\cap Z_0=V_1$. Then, for all $z\in V_L\cap Z_0$,  by the definition of the set $V_1$, there exists $x\in X$ and $y\in  \Phi(x_0,0_Z)-L(x_0)+k_0-\inte K $ such that $ \Phi(x,z)-L(x)\in y-K . $
 It then  follows that 
$$\Phi(x,z)-L(x)\in\Phi(x_0,0_Z)-L(x_0)+k_0-\inte K-K = y_L-\inte K\subset y_L-K,$$
yielding   $\Phi(x,z)-L(x)\leqq_K y_L$.  Consequently,  $(C_1)$ holds. 

$(\epsilon)$ Assume that $(C_6)$ holds.
 By Proposition \ref{pro_incluepi},  to prove \eqref{eq_18d}, it is sufficient to show that 
$\epi\Phi(.,0_Z)^\ast    \subset  \mathcal{M},$ 
of which the  proof goes along the line as  that  of   Theorem \ref{thm_nonasymptotic_representing_epi_2}.

Firstly, take $(\bar L,\bar y)\in \epi \Phi(.,0_Z)^*$. Then  \eqref{eq_31bbb}  holds  by \eqref{epiF*}.  Let us set 
\begin{equation*}
	\Delta' _{\bar L}:=\bigcup_{(x,z)\in \dom \Phi}
		\Bigg( \Big(\bar L(x)-\Phi(x,z)-K\Big)\times \Big(z-\pi(\dom \Phi)\Big) \Bigg).
\end{equation*}
\indent $\bullet$ As $\Phi$ is $K$-convex,  it is easy to check that $\Delta'_{\bar L} \subset Y\times Z$ is  convex.    Moreover, as the sets  $\dom\Phi$, $\inte K$, and $\inte \pi (\dom \phi) $ are nonempty, $\inte \Delta'_{\bar L} $ is nonempty, as well. 

$\bullet$ We show that $(\bar y,0_{Z})\notin \inte\Delta' _{\bar L}$. 
Indeed, if $(\bar y,0_{Z})\in \inte\Delta'_{\bar L}$ then, by the same argument as in the proof of Theorem \ref{thm_nonasymptotic_representing_epi_2}, there exists $\bar k\in\inte K$ such that 
 $(\bar y+\bar{k},0_{Z})\in \Delta' _{\bar L}$. 
Hence, there is  $(\bar x,\bar z) \in \dom \Phi$ such that $\bar y+\bar{k}\in {\bar L} (\bar{x})-\Phi(\bar x, \bar z)-K$ and $0_Z\in \bar z-\pi(\dom \Phi)$.  Taking  $(C_6)$ into account,  one gets    $\Phi (\bar x,0_Z) \leqq_K  \Phi(\bar x,\bar z)$. 
Consequently,  $\bar y+\bar k \in {\bar L}(\bar x)-\Phi(\bar x,0_Z)-K,$ yielding  $\bar y \in {\bar L} (\bar x)-\Phi(\bar x,0_Z)-\inte K$, which           contradicts \eqref{eq_31bbb}.

$\bullet$ By  the   separation theorem (\cite[Theorem 3.4]{Rudin91}) applying  to the point $(v,0_{Z})$ and   the convex set $\inte\Delta' _{\bar L}$ in $Y \times Z$,   there is  $(y^*_0, z^*_0)\in Y^{*}\times Z^* \setminus \{ (0_{Y^*}, 0_{Z^*})\}$  such that
\begin{equation}  \label{abcd20}
	y^*_0(v) < y^*_0(y) + z^*_0(z),\quad \forall (y,z)\in \inte\Delta'_{\bar L}.
\end{equation}
\indent $\bullet$ We now  show  that
\begin{equation}  \label{eqaaaa}
y^*_0(k')<0, \quad \forall k' \in \inte K.
\end{equation}
Take $k'\in\inte K$. 
On the one hand, it follows from the last part of $(C_6)$ that there exists $(\widehat x,\widehat z)\in \dom \Phi$ satisfying $0_Z\in \widehat z- \inte \pi(\dom \Phi)$. On the other hand, according to Lemma \ref{pro_1a}(i), there exists $\mu>0$ such that and hence $v-\mu k'\in \bar  L(\widehat x)-\Phi(\widehat x, \widehat z)-\inte K$.
So, the set
\begin{equation*}
W:=\Big(\bar L(\widehat x) - \Phi(\widehat x,\widehat z) -\inte K\Big)\times \Big(\widehat z -\inte \pi(\dom \Phi)\Big)
\end{equation*}
is a neighborhood of $(v-\mu k',0_Z)$. It is easy to see that $W\subset \Delta'_{\bar L} $,  yielding $(v-\mu k',0_Z)\in\inte \Delta'_{\bar L}$. Applying \eqref{abcd20},  one has  $y^*_0(v)<y^*_0(v-\mu k')+z^*_0(0_Z)$,  which yields    \eqref{eqaaaa}. 

$\bullet$ Now, take $k_0 \in \inte K$ such that $y^\ast_0(k_0)=-1$ and define  $T\in \L(Z,Y)$  by $ T(z)=-z^{\ast }_0(z) k_{0}$ for any $z \in Z$. 
Using the same argument as in the proof of Theorem \ref{thm_nonasymptotic_representing_epi_2} (see \eqref{abcd22}), we can  show that 
$$\bar y\notin y+T(z)    - \inte K,  \,\quad\forall (y,z)\in \Delta'_{\bar L}$$
which, together the fact that $(\bar L(x)-\Phi(x,z),z)\in \Delta'_{\bar L}$ for all $(x,z)\in \dom\Phi$, yields
$$\bar y\notin \bar L(x)-\Phi(x,z) +T (z)-\inte K,\quad \forall (x,z)\in \dom \Phi.$$
Again,  \eqref{epiF*}    (applying to the mapping $\Phi$) yields   $(\bar L,T,\bar y)\in \epi \Phi^*$, or equivalently, $(\bar L,\bar y)\in \epi \Phi^*(.,T)$.    The    inclusion 
$\epi\Phi(.,0_Z)^\ast   \subset \mathcal{M}    $  has been proved, and so \eqref{eq_18d} holds.

$\bullet$ The proof of \eqref{eq_18ddlbis}  (under extra assumption $(C_0)$) goes on the same way as  last part  the proof of     Theorem \ref{thm_nonasymptotic_representing_epi_2}.  \qed


\section{\small Proof of Theorem \ref{thm_nonasymptotic_representing_epi}} \label{A-C}   

Take $k_0 \in \inte K$,   $(\bar L,\bar y)\in \epi \Phi(.,0_Z)^*$ and  consider the set    
\begin{equation*}
	\Delta'' _{\bar L}:=\bigcup_{(x,z)\in \dom \Phi}
		\Big[ \Big(L(x)-\Phi(x,z)-K\Big)\times \Big(z-S\Big) \Big]. 
\end{equation*}

By a similar argument asas in Appendix \ref{A-B}$(\epsilon)$, using  $\Delta'' _{\bar L} $   instead of 
 $\Delta^\prime_{\bar L}$, we  can establish   $T\in\L_\Phi$   such that $(\bar L, \bar y) \in \epi \Phi^\ast (\cdot, T)$, and we get 
 \eqref{eq_18d}.
 
 For the proof of    \eqref{eq_18ddlbis},  take $(\tilde x,0_Z)\in \dom \Phi$. For any  $s\in S$    and $\nu>0$, 
one has $-\nu s\in -S=0_Z-S$, and hence, $(\bar L(\tilde x)-\Phi(\tilde x), -\nu z)\in \Delta''_{\bar L}$.  The same argument as the last part of  the proof of Theorem \ref{thm_nonasymptotic_representing_epi_2} will  leads to $T \in \L_+(S, K)$.  
 \qed 
 
\section{\small Proof of Lemma  \ref{lem:6.1_nwewewE}} \label{A-D}   

$(i)$ Take $L\in \L(X,Y)$ and $T:=(T_1,T_2)\in \L(W,Y)\times \L(Z,Y)$, one has
\begin{align*}
\Phi_1^*(L,T)
&=\wsup\left\{
\begin{array}{r}
L(x)+T_1(w)+T_2(z)
-F(x)+\kappa ( H(x)+w) 
:\;
 x\in C,\\ H(x)+w\in \dom \kappa,\; G(x)+z\in -S
\end{array}
\right\}\\
&=\wsup\left\{
\begin{array}{r}
L(x)+T_1(u-H(x))+T_2(s-G(x))
-F(x)+\kappa ( u) :\\
 x\in C,\; u\in \dom \kappa, \; s\in -S
\end{array}
\right\}\\
&=\wsup\left\{
\begin{array}{r}
L(x)-F(x)-(T_1\circ H)(x)-(T_2\circ G)(x)
+T_1(u)-\kappa(u) +T_2(s):\\
\quad  x\in C,\; u\in \dom \kappa,\; s\in -S.
\end{array}\right\}\\
&=\wsup\big[(L-F-T_1\circ H-T_2\circ G)(C)+(T_1-\kappa)(\dom \kappa)+ T_2(-S)\big]\\
&=
\wsup\big [ \wsup[(L-F-T_1\circ H-T_2\circ G)(C)]+\wsup[(T_1-\kappa)(\dom \kappa)]+\\
& \hskip4cm+\wsup[T_2(-S)]\big]\qquad \textrm{(by Proposition \ref{pro_decomp}(iii))}\\
&=\wsup\big[(F+T_1\circ H+T_2\circ G+I_{C})^*(L)+\kappa^\ast(T_1)+I^*_{-S}(T_2)\big]\\
&=(F+T_1\circ H+T_2\circ G+I_{C})^*(L)\uplus\kappa^\ast(T_1)\uplus I^*_{-S}(T_2). 
\end{align*}

Assume further that $T_2\in \L_+(S,K)$. Then, $I^\ast_{-S}(T_2)=\wsup [T_2(-S)]=\wsup \{0_Y\}=-\bd K$  
 (applying Proposition \ref{pro_decomp} (vii)  to  $N=T_2(-S)$ and $M=\{0_Y\}$), and hence, 
 
\vskip-0.4cm 
\begin{align*}
\Phi_1^*(L,T)&=(F+T_1\circ H_1+T_2\circ H_2+I_{C})^*(L)\uplus\kappa^\ast(T_1)\uplus (-\bd K)\\
&=(F+T_1\circ H_1+T_2\circ H_2+I_{C})^*(L)\uplus\kappa^\ast(T_1),   \ \ \ \ \textrm{(by Proposition \ref{prop_1ab}(i))}.
\end{align*}
$(ii)$ We just prove the first equality, the second one can be obtained by the same argument. 
Take  $T:=(T_1,T_2)\in \L(W,Y)\times \L(Z,Y)$, one has 

 
 \begin{align*}
&(L,y)\in \epi \Phi^\ast(.,T)\\
&\Longleftrightarrow y\in (F+T_1\circ H+T_2\circ G+I_{C})^*(L)\uplus\kappa^\ast(T_1)\uplus I^*_{-S}(T_2)+K\quad \textrm{(by \eqref{eq:45nwewE})}\\
&\Longleftrightarrow \exists U\in \P_p(Y): 
y\in U\uplus\kappa^\ast(T_1)\uplus I^*_{-S}(T_2),\; (F+T_1\circ H+T_2\circ G+I_{C})^*(L) \preccurlyeq_K U \
&\hskip9cm\textrm{(by Proposition \ref{prop_1ab}(v))}\\
&\Longleftrightarrow   \exists U\in \P_p(Y): 
y\in U\uplus\kappa^\ast(T_1)\uplus I^*_{-S}(T_2), \textrm{ and }
 (L,U\uplus\kappa^\ast(T_1)\uplus I^*_{-S}(T_2))=\\
& \hskip0.7cm (L,U)\boxplus (0_\L,\kappa^\ast(T_1)\uplus I^*_{-S}(T_2))\in \exepi  (F+T_1\circ H+T_2\circ G+I_{C})^* \boxplus (0_\L,\kappa^\ast(T_1)\uplus I^*_{-S}(T_2)) 
\\
&\Longleftrightarrow (L,y)\in \Psi(\exepi (F+T_1\!\circ\! H_1+T_2\!\circ\! H_2+I_{C})^\ast\boxplus (0_\L, \kappa^\ast(T_1)\uplus I_{-S}^\ast(T_2)).
\end{align*}
which means, 
\begin{equation}
\label{eq:46_nwewewE}
\epi \Phi^\ast(.,T)
=\Psi(\exepi (F+T_1\!\circ\! H_1+T_2\!\circ\! H_2+I_{C})^\ast\boxplus (0_\L, \kappa^\ast(T_1)\uplus I_{-S}^\ast(T_2)).
\end{equation}
Next, if $T\notin \dom \kappa^*\times \L_+^w(S,K)$ then $\kappa^\ast(T_1)=\{+\infty_Y\}$ or $I^\ast_{-S}(T_2)=\{+\infty_Y\}$ (see \eqref{domI*}) . This, together with \eqref{eq:45nwewE}, yields $\Phi_1^*(L,T)=\{+\infty_Y\}$ for all $L\in \L(X,Y)$. So, 
$\L_{\Phi_1}\subset \dom \kappa^*\times \L_+^w(S,K),$
and the desired equality in $(\mathrm {ii})$ follows from \eqref{eq:46_nwewewE}  and the fact that  
$\epi \Phi^*_1(.,T)=\emptyset$ whenever $T\notin \L_{\Phi_1}$. 
\qed


}

\end{document}